\documentclass[a4paper, 11pt, reqno]{amsart}
\usepackage{hyperref}
\usepackage{times}
\usepackage{lhelp}
\usepackage[T1]{fontenc}
\usepackage{amsmath}
\usepackage{pb-diagram}
\usepackage{amsfonts}
\usepackage{amsthm}
\usepackage{layout}
\usepackage{dsfont}
\usepackage[dvips]{color}
\input{amssym.def}
\input{amssym.tex}

\input colordvi

\newcommand\new[1]{}
\usepackage{enumerate}
\usepackage{verbatim}

\input{colordvi}

\theoremstyle{plain}
\newtheorem{theorem}{Theorem}[section]

\theoremstyle{remark}
\newtheorem{rem}[theorem]{Remark}

\newtheorem{ex}[theorem]{Example}

\theoremstyle{definition}
\newtheorem{defi}[theorem]{Definition}

\theoremstyle{plain}
\newtheorem{cor}[theorem]{Corollary}
\newtheorem{lem}[theorem]{Lemma}
\newtheorem{prop}[theorem]{Proposition}

\newtheorem{Assumption}{Assumption}

\numberwithin{equation}{section}

\def\imbed{\hookrightarrow}

\newcommand{\rf}[1]{(\ref{#1})}
\newcommand\norm[1]{\left|\!\left| #1\right|\!\right|}

\newcommand\BIP{\operatorname{BIP}}

\newcommand\sgn{\operatorname{sgn}}

\newcommand\Vertt{|\!|}

\newcommand{\U}{\mathcal{U}}
\newcommand{\M}{\mathcal{M}}

\newcommand{\tv}{\tilde{v}}
\newcommand{\tq}{\tilde{q}}
\newcommand{\tz}{\tilde{z}}
\newcommand{\tX}{\tilde{X}}

\newcommand{\tf}{\tilde{f}}

\newcommand{\lab}[1]{\label{#1}}
\newcommand{\cit}[1]{\cite{[#1]}}

\def\frakJ{\mathfrak{J}}

\def\R{{\mathds R}}

\def\Re{\mbox{Re\,}}

\newcommand{\tFil}{\tilde{\mathcal{F}}}
\newcommand{\tProb}{\tilde{\mathds{P}}}
\newcommand{\tExp}{\tilde{\mathds{E}}}
\newcommand{\Exp}{\mathds{E}}
\newcommand{\tOmega}{\tilde{\Omega}}

\newcommand{\Prob}{\mathds{P}}

\newcommand{\tlambda}{\tilde{\lambda}}
\newcommand{\tM}{\tilde{M}}
\newcommand{\Y}{\mathcal{Y}}

\newcommand{\Lin}{\mathcal{L}}
\newcommand{\abs}[1]{\left\vert#1\right\vert}
\newcommand{\set}[1]{\left\{#1\right\}}
\newcommand{\seq}[1]{\left<#1\right>}

\newcommand{\bE}{\mathbf{E}}
\newcommand{\bH}{\mathbf{H}}

\newcommand{\B}{\mathcal{B}}
\newcommand{\A}{\mathcal{A}}

\newcommand{\Fil}{\mathcal{F}}

\newcounter{notectr}
\renewcommand{\thenotectr}{\arabic{notectr}}
{\refstepcounter{notectr}
\noindent
{\it  Note {\thenotectr}}%:
\em \normalsize}{}%

\newcounter{dig}
{\bigskip \noindent\large \sl Digression \arabic{dig}
 \small\rm }% %%%\footnotesize
{\stepcounter{dig}{\hfill}
}
\addtocounter{dig}{+1}

\newcounter{remctr}
\renewcommand{\theremctr}{\arabic{remctr}}
{\refstepcounter{remctr}
\noindent
\it  Remarks {\theremctr}:%
\normalsize\em}{}%

\newcounter{condctr}
\renewcommand{\thecondctr}{\Alph{condctr}}
\newcounter{subcondctr}[condctr]

{\refstepcounter{condctr}
\smallskip \noindent
\large \sc  Condition {\thecondctr}:%
\normalsize\em\smallskip}{}%

\title[Relaxed control of stochastic evolution equations]{Optimal relaxed control of dissipative stochastic partial differential equations in Banach spaces}
\author[Z.~Brzezniak]{Zdzislaw Brzezniak}
\address{Department of Mathematics\\ University of York\\ Heslington, York UK\\ YO10 5DD}
\email{zb500@york.ac.uk}

\author[R.~Serrano]{Rafael Serrano}
\address{Department of Mathematics\\ University of York\\ Heslington, York UK\\ YO10 5DD}
\email{rasp500@york.ac.uk}
\urladdr{http://www-users.york.ac.uk/~rasp500}
\thanks{The second named author is supported by a Dorothy Hodgkin Postgraduate Award}

\date{\today}

\begin{document}
%\layout
\begin{abstract}
We study an optimal relaxed control problem for a class of semilinear stochastic PDEs on Banach spaces perturbed by multiplicative noise and driven by a cylindrical Wiener process. The state equation is controlled through the nonlinear part of the drift coefficient which satisfies a dissipative-type condition with respect to the state variable. The main tools of our study are the factorization method for stochastic convolutions in UMD type-2 Banach spaces and certain compactness properties of the factorization operator and of the class of Young measures on Suslin metrisable control sets.
\end{abstract}

\bibliographystyle{amsalpha}

\maketitle

%\tableofcontents

\section{Introduction}
Let $M$ be a separable metric space, $B$ a Banach space and $(\Omega,\Fil,\Prob)$ a probability space. The object of this paper is to study the optimal control problem of minimizing a finite-horizon cost functional of the form
\begin{equation}\label{cf}
J(X,u)=\Exp\left[\int_0^T h(s,X(s),u(s))\,ds+\varphi(X(T))\right],
\end{equation}
where $u(\cdot)$ is a $M-$valued control process and $X(\cdot)$ is the solution (in a sense we will specify later) of the controlled semilinear stochastic evolution equation of the form
\begin{equation}\label{eq0}
\begin{split}
dX(t)+AX(t)\,dt&=F(t,X(t),u(t))\,dt+G(t,X(t))\,dW(t), \ \ \ t\in
[0,T]\\
X(0)&=x_0\in B.
\end{split}
\end{equation}
Here $-A$ is the generator of a $C_0-$semigroup on $B$ and $W(\cdot)$ is a cylindrical Wiener process defined on $(\Omega,\Fil,\Prob).$ Our model equation is the following controlled stochastic partial differential equation of reaction-diffusion type with multiplicative space-time white noise on $[0,T]\times (0,1),$
\begin{align*}
\frac{\partial X}{\partial t}(t,\xi)&=\frac{\partial^2X}{\partial\xi^2}(t,\xi)+f(X(t,\xi),u(t))+g(X(t,\xi))\,\frac{\partial w}{\partial t}(t,\xi), \ \ \text{on }
[0,T]\times (0,1),\\
X(t,0)&=X(t,1)=0,\\
X(0,\cdot)&=x_0(\cdot),
\end{align*}
where $g:\R\to\R$ is continuous and bounded and the reaction term $f:\R\times M\to\R$ is continuous and satisfies a polynomial-growth condition of the form
\begin{equation}\label{pgc}
f(x+y,u)\sgn x\le -k_1\abs{x}+k_2\abs{y}^m+\eta(t,u)
\end{equation}
where $m\geq 1,$ $k_1\in\R, \ k_2\geq 0$ and $\eta:[0,T]\times M\to[0,+\infty]$ is a measurable mapping, possibly inf-compact with respect to $u$ (see Definition \ref{infcompact} below). In the applications, $X(t,\xi)$ represents the concentration, density or temperature of a certain substance. Our aim is to be able to consider cost functionals that regulate this quantity on a finite number of points $\zeta_1,\ldots,\zeta_n$ distributed over the interval $(0,1),$ for instance, when the running cost function $h$ has of the form
\[
h(t,x,u)=\phi(t,x(\zeta_1),\ldots,x(\zeta_n),u).
\]
Such running costs, however, require $X(t,\xi)$ to be continuous with respect to the space variable, which suggests that we study the controlled equation on the Banach space $\mathcal{C}([0,1])$ of real-valued continuous functions on $[0,1].$ Moreover, the growth condition (\ref{pgc}) can be conveniently formulated as a dissipative-type condition on this Banach space.

It is well known that, when no special conditions on the the dependence of the non-linear term $F$ with respect to the control variable are assumed,  in order to prove existence of an optimal control it is necessary to extend the original control system to one that allows for control policies whose instantaneous values are probability measures on the control set. Such control policies are known as \emph{relaxed controls.}

This technique of measure-valued convexification of nonlinear systems has a long story starting with L.C.Young \cite{young1,youngbook} and J.Warga \cite{warga1,wargabook} and their work on variational problems and existence of optimal controls for finite-dimensional systems. The use of relaxed controls in the context of evolution equations in Banach spaces was initiated by Ahmed \cite{ahmed1} and Papageorgiou \cite{papag2,papag1} (see also \cite{papag5}) who considered controls that take values in a Polish space. More recently, optimal relaxed control of PDEs has been studied by Lou in \cite{lou1,lou2} also with Polish control set.

Under more general topological assumptions on the control set, Fattorini also employed relaxed controls in \cite{fattorini2} and \cite{fattorini1} (se also \cite{fattorinibook}) but at the cost of working with merely finitely additive measures instead of $\sigma-$additive measures.

In the stochastic case, relaxed control of finite-dimensional stochastic systems goes back to Fleming and Nisio \cite{fleming1,flenisio}. Their approach was followed extensively in \cite{elkarouietal} and \cite{hausslep}, where the control problem was recast as a martingale problem. The study of relaxed control for stochastic PDEs seems to have been initiated by Nagase and Nisio in \cite{nagasenisio} and continued by Zhou in \cite{zhou1}, where a class of semilinear stochastic PDEs controlled through the coefficients of an elliptic operator and driven by a $d-$dimensional Wiener process is considered, and in \cite{zhou2}, where controls are allowed to be space-dependent and the diffusion term is a first-order differential operator driven by a one-dimensional Wiener process.

In \cite{gatsob}, using the semigroup approach, Gatarek and Sobczyk extended some of the results described above to Hilbert space-valued controlled diffusions driven by a trace-class noise. The main idea of their approach is to show compactness of the space of admissible relaxed control policies by the factorization method introduced by Da Prato, Kwapien and Zabczyk (see \cite{dapratoetal}). Later, in \cite{sri} Sritharan studied optimal relaxed control of stochastic Navier-Stokes equations with monotone nonlinearities and Lusin metrisable control set. More recently, Cutland and Grzesiak combined relaxed controls with nonstandard analysis techniques in \cite{cutgr1,cutgr2} to study existence of optimal controls for 3 and 2-dimensional stochastic Navier-Stokes equations respectively.

In this paper, we consider a control system and use methods that are similar to those of \cite{gatsob}. However, our approach allows to consider controlled processes with values in a larger class of state spaces, which permits to study running costs that are not necessarily well-defined in a Hilbert-space framework. Moreover, we consider controlled equations driven by cylindrical Wiener process, which includes the case of space-time white noise in one dimension, and with a drift coefficient that satisfies a dissipative-type condition with respect to the state-variable. In addition, the control set is assumed only metrisable and Suslin.

Let us briefly describe the contents of this paper. In section 2 we recall the notion of stochastic relaxed control and its connection with random Young measures, we define the stable topology and review some relatively recent results on (flexible) tightness criteria for relative compactness in this topology. Next, we introduce the factorization operator as the negative fractional power of a certain abstract parabolic operator associated with a Cauchy problem on UMD spaces and some of its smoothing and compactness properties. Then, we review some basics results on the factorization method for stochastic convolutions in UMD type-2 Banach spaces.

In section 3 we reformulate the control problem as a relaxed control problem in the weak stochastic sense and prove existence of optimal weak relaxed controls for a class of dissipative stochastic PDEs. Finally, we illustrate this result with examples that cover a class of stochastic reaction-diffusion equations (driven by space-time multiplicative white noise in dimension 1) and include the case of space-dependant control.

\noindent\textbf{Notation.} Let $\mathcal{O}$ be a bounded domain in $\mathds{R}^d.$ For $m\in\mathds{N}$ and $p\in[1,\infty],$ $W^{m,p}(\mathcal{O})$ will denote the usual Sobolev space, and for $s\in\R,$ $H^{s,p}(\mathcal{O})$ will denote the space defined as
\[
H^{s,p}(\mathcal{O}):=
\left\{
  \begin{array}{ll}
    W^{m,p}(\mathcal{O}), & \ \hbox{if  \ \ $m\in\mathds{N}$;}\\
    \left[W^{k,p}(\mathcal{O}),W^{m,p}(\mathcal{O})\right]_\delta, & \ \text{if} \ \ s\in(0,\infty)\setminus\mathds{N},
  \end{array}
\right.
\]
where $[\cdot,\cdot]_\delta$ denotes complex interpolation and $k,m\in\mathds{N}, \ \delta\in(0,1)$ are chosen to satisfy $s=(1-\delta)k+\delta m$ (see e.g. \cite{triebel}).

\section{Preliminaries}

\subsection{Relaxed controls and Young measures}
We start by recalling the definition of stochastic relaxed control and its connection with random Young measures. Throughout, $M$ denotes a Hausdorff topological space (the control set), $\B(M)$ denotes the Borel $\sigma-$algebra on $M$ and $\mathcal{P}(M)$ denotes the set of probability measures on $\B(M)$ endowed with the $\sigma-$algebra generated by the projection maps
\[
\pi_C:\mathcal{P}(M)\ni q\mapsto q(C)\in [0,1], \ \ \ C\in\B(M).
\]
\begin{defi}
Let $(\Omega,\Fil,\Prob)$ be a probability space. A $\mathcal{P}(M)-$valued process $\set{q_t}_{t\geq 0}$ is called a \emph{stochastic relaxed control} (or relaxed control process) on $M$ if and only if the map
\[
[0,T]\times\Omega\ni (t,\omega)\mapsto q_t(\omega,\cdot)\in\mathcal{P}(M)
\]
is measurable. In other words, a stochastic relaxed control is a measurable $\mathcal{P}(M)-$valued process.%We denote by $\mathcal{R}(0,T;M)$ the set of (classes of equivalence of)\linebreak stochastic relaxed controls on $M.$
\end{defi}
\begin{defi}
Let $l$ denote the Lebesgue measure on $[0,T]$ and let $\lambda$ be a bounded nonnegative $\sigma-$additive measure on $\B\left(M\times [0,T]\right).$ We say that $\lambda$ is a \emph{Young measure} on $M$ if and only if $\lambda$ satisfies
\begin{equation}\label{youngmeasure}
\lambda(M\times D)=l(D), \ \ \ \text{for all }D\in \B([0,T]).
\end{equation}
We denote by $\Y(0,T;M),$ or simply $\Y,$ the set of Young measures on $M.$
\end{defi}
\begin{lem}[\textbf{Disintegration of `random' Young measures}]\label{Lem:disrym}
Let $(\Omega,\Fil,\Prob)$ be a probability space and let $M$ be a Radon space. Let $\lambda:\Omega\to\Y(0,T;M)$ be such that, for every $J\in\B(M\times [0,T]),$ the mapping
\[
\Omega\ni \omega\mapsto \lambda(w)(J)\in [0,T]
\]
is measurable. Then there exists a stochastic relaxed control $\set{q_t}_{t\geq 0}$ on $M$ such that for $\Prob-$a.e. $\omega\in\Omega$ we have
\begin{equation}\label{disrym}
\lambda(\omega,C\times D)=\int_D q_t(\omega,C)\,dt, \ \ \ \ \text{for all } \ \ C\in\B(M), \ D\in\B([0,T]).
\end{equation}
\end{lem}
\begin{proof}
Define the measure $\mu$ on $\B(M)\otimes\B([0,T])\otimes\Fil$ by
\[
\mu(du,dt,d\omega):=\lambda(\omega)(du,dt)\,\Prob(d\omega),
\]
that is,
\begin{equation}\label{mulambda}
\mu(C\times D\times E)=\Exp\left[1_E\lambda(C\times D)\right], \ \ \ C\in \B(M), \ D\in \B([0,T]), \ E\in\Fil.
\end{equation}
Notice that the marginals of $\mu$ on $\Fil\otimes\B([0,T])$ coincide with the product measure $d\Prob\otimes dt.$ Hence, as $M$ is a Radon space, by the Disintegration Theorem (cf. existence of conditional probabilities, see e.g. \cite{val73}), % or Satz 5.3.21 of \cite{GaStu})
there exists a mapping
\[
\tq:[0,T]\times\Omega\times\B(M)\to [0,1]
\]
satisfying
\begin{equation}\label{mutq}
\mu(C\times J)=\int_{J}\tq(t,\omega,C)\,d\Prob\otimes dt, \ \ \ C\in\B(M), \ J\in\Fil\otimes B([0,T]),
\end{equation}
and such that for every $C\in \B(M),$ the mapping
\[
[0,T]\times\Omega\ni(t,\omega)\mapsto \tq(t,\omega,C)\in [0,1]\]
is measurable and, for almost every $(t,\omega)\in[0,T]\times\Omega$, $\tq(t,\omega,\cdot)$ is a Borel probability measure on $\B(M).$ Therefore
\[
q:[0,T]\times\Omega\ni(t,\omega)\to\tq(t,\omega,\cdot)\in\mathcal{P}(M)
\]
is a stochastic relaxed control. Moreover, by (\ref{mulambda}), (\ref{mutq}) and Fubini's Theorem we have
\[\int_E\lambda(\omega)(C\times D)\,\Prob(d\omega)=\int_E\int_D q(t,\omega)(C)\,dt\,\Prob(d\omega)\]
for every $E\in\Fil$ and $C\in \B(M), \ D\in \B([0,T]).$ Hence, (\ref{disrym}) follows.
\end{proof}
\begin{rem}
We will frequently denote the disintegration (\ref{disrym}) by $\lambda(du,dt)=q_t(du)\,dt.$
\end{rem}

\subsection{Stable topology and tightness criteria}
\begin{defi}
The \emph{stable topology} on $\Y(0,T;M)$ is the weakest topology on $\Y(0,T;M)$ for which the mappings %smallest,weakest
\[
\Y(0,T;M)\ni\lambda\mapsto\int_D\int_M f(u)\,\lambda(du,dt)\in\R
\]
are continuous, for every $D\in\B([0,T])$ and $f\in\mathcal{C}_b(M).$
\end{defi}
The stable topology was studied under the name of \emph{ws}-topology in \cite{schael}. There it was proved that if $M$ is separable and metrisable, then the stable topology coincides with the topology induced by the \emph{narrow topology}. The case of $M$ Polish (i.e. separable and completely metrisable) was studied in \cite{JaMe}. A comprehensive overview on the stable topology for a more general class of Young measures under more general topological conditions on $M$ can be found in \cite{cadfval}.

\begin{rem}
It can be proved (see e.g. Remark 3.20 in \cite{crauel}) that if $M$ is separable and metrisable, then $\lambda:\Omega\to\Y(0,T;M)$ is measurable with respect to the Borel $\sigma-$algebra generated by the stable topology iff for every $J\in\B(M\times [0,T])$ the mapping
\[
\Omega\ni \omega\mapsto \lambda(w)(J)\in [0,T]
\]
is measurable. This will justify addressing the maps considered in Lemma \ref{disrym} as random Young measures.
\end{rem}
A class of topological spaces that will be particularly useful for our purposes is that of Suslin space.
\begin{defi}
A Hausdorff topological space $M$ is said to be \emph{Suslin} if there exist a Polish space
$S$ and a continuous mapping $\varphi:S\to M$ such that $\varphi(S)=M.$
\end{defi}
\begin{rem}
If $M$ is Suslin then $M$ is separable and Radon, see e.g. \cite[Chapter II]{schwartz1}. In particular, Lemma \ref{Lem:disrym} applies.
\end{rem}
We will be mainly interested in Young measures on metrisable Suslin control sets. This class of Young measures has been studied in \cite{balder2} and \cite{deFitte1}.%See also \cite{cadfval}.
\begin{prop}\label{stsuslinmet}
Let $M$ be metrisable (resp. metrisable Suslin). Then $\Y(0,T;M)$ endowed with the stable topology is also metrisable (resp. metrisable Suslin).
\end{prop}
\begin{proof}
For the metrisability part, see Proposition 2.3.1 in \cite{cadfval}. For the Suslin part, see Proposition 2.3.3 in \cite{cadfval}.
\end{proof}
The notion of tightness for Young measures that we will use has been introduced by Valadier \cite{val90} (see also \cite{crauel}). Recall that a set-valued function $[0,T]\ni t\mapsto K_t\subset M$ is said to be \emph{measurable} if and only if
\[
\set{t\in [0,T]: K_t\cap U\neq\emptyset}\in\B([0,T])
\]
for every open set $U\subset M.$
\begin{defi}
We say that $\frakJ\subset\Y(0,T;M)$  is \emph{flexibly tight} if,
for each $\varepsilon>0,$ there exists a measurable set-valued mapping $[0,T]\ni t\mapsto K_t\subset M$ such that $K_t$ is compact for all $t\in [0,T]$ and
\[
\sup_{\lambda\in\frakJ}\int_0^T\!\int_M \mathbf{1}_{K_t^c}(u)\,\lambda(du,dt)<\varepsilon
\]
\end{defi}
%In other words $\frakJ$ is strictly tight if and only if the marginals on $M$ of the measures in $\frakJ$ form a tight subset of $\mathcal{P}(M).$
%Moreover, if $J$ is strictly tight and $M$ is Suslin regular, then $J$ is tight in $\M^{+}(M\times [0,T])$ in the usual sense (why?)
In order to give a characterization of flexible tightness we need the notion of an inf-compact function,
\begin{defi}\label{infcompact}
A function $\eta:M\to[0,+\infty]$ is called \emph{inf-compact} iff the level sets
\[
\set{\eta\le R}=\set{u\in M:\eta(u)\le R}
\]
are compact for all $R\geq 0.$
\end{defi}
Observe that, since $M$ is Hausdorff, for every inf-compact function $\eta$ the level sets $\set{\eta\le R}$ are closed. Therefore, every inf-compact function is lower semi-continuous and hence Borel-measurable (see e.g. \cite{kallen}).
\begin{ex}\label{infcompactex}
Let $(V,\abs{\cdot}_V)$ be a reflexive Banach space \emph{compactly} embedded into another Banach space $(M,\abs{\cdot}_M),$ and let $a:\R^{+}\to\R^{+}$ be strictly increasing and continuous. Then the map $\eta:M\to[0,+\infty]$ defined by
\[
\eta(u):=
\left\{
  \begin{array}{ll}
    a(\abs{u}_V), & \hbox{if} \ u\in V\\
    +\infty, & \hbox{else.}
  \end{array}
\right.
\]
is inf-compact.
\begin{proof}[Proof of Example \ref{infcompactex}]
Since $a(\cdot)$ is increasing, we only need to show that the closed unit ball $D$ in $V$ is compact in $M.$ Let $(u_n)_{n}$ be a sequence in $D.$ Since the embedding $V\hookrightarrow M$ is compact, there exist a subsequence, which we again denote by $(u_n)_{n},$ and $u\in M$ such that $u_n\to u$ in $M$ as $n\to\infty.$ Hence, if $C$ is a constant such that $\abs{v}_M\le C\abs{v}_V, \ v\in V,$ and $\varepsilon>0$ is fixed we can find $\bar m\in\mathds{N}$ such that
\begin{equation}\label{ep1}
\abs{u_n-u}_M<\frac{\varepsilon}{1+C}, \ \ \forall n\geq \bar m.
\end{equation}
Now, since $V$ is reflexive, by the Banach-Alaoglu Theorem there exists a further subsequence, again denoted by $(u_n)_{n},$ and $\bar u\in V$ such that $u_n\to \bar u$ weakly in $V$ as $n\to\infty.$ In particular, this implies
\[
\bar u\in\overline{\{u_{\bar{m}},u_{\bar{m}+1},\ldots\}}^{w}
\subset \overline{{\rm co}\{u_{\bar{m}},u_{\bar{m}+1},\ldots\}}^{w}
\]
where ${\rm co}(\cdot)$ and $\overline{\,\cdot\,}^{w}$ denote the convex hull and weak-closure in $V$ respectively. By Mazur Theorem (see e.g. \cite[Theorem 2.5.16]{meg}), we have
\[
\overline{{\rm co}\{u_{\bar{m}},u_{\bar{m}+1},\ldots\}}^{w}
=\overline{{\rm co}\{u_{\bar{m}},u_{\bar{m}+1},\ldots\}}.
\]
Hence, there exist an integer $\bar{N}\geq 1$ and $\{\alpha_0,\ldots,\alpha_{\bar{N}}\}$ with $\alpha_i\geq 0,$ $\sum_{i=0}^{\bar{N}}\alpha_i=1,$ such that
\begin{equation}\label{ep2}
\Bigl|\sum_{i=0}^{\bar{N}}\alpha_i u_{\bar m+i}-\bar u\Bigr|_{V}<\frac{\varepsilon}{1+C}.
\end{equation}
By (\ref{ep1}) and (\ref{ep2}) it follows that
\begin{align*}
\abs{u-\bar u}_M&\le \Bigl|u-\sum_{i=1}^{\bar{N}}\alpha_i u_{\bar m+i}\Bigr|_{M}+\Bigl|\sum_{i=0}^{\bar{N}}\alpha_i u_{\bar m+i}-\bar u\Bigr|_{M}\\
&\le \Bigl|\sum_{i=0}^{\bar{N}}\alpha_i (u-u_{\bar m+i})\Bigr|_{M}+C\Bigl|\sum_{i=0}^{\bar{N}}\alpha_i u_{\bar m+i}-\bar u\Bigr|_{V}\\
&\le \sum_{i=0}^{\bar{N}}\alpha_i\abs{u-u_{\bar m+i}}_{M}+\frac{C\varepsilon}{1+C}\\
&<\varepsilon.
\end{align*}
Since $\varepsilon>0$ is arbitrary, we infer that $u=\bar u\in V.$ Therefore, $D$ is sequentially compact in $M,$ and the desired result follows.
\end{proof}
\end{ex}

\begin{theorem}[\textbf{Equivalence Theorem for flexible tightness}]\label{flexiblytight}
Let $\frakJ\subset\Y(0,T;M).$ Then the following conditions are equivalent
\begin{enumerate}[{\rm (a)}]
  \item $\frakJ$ is flexibly tight
  \item There exists a measurable function $\eta:[0,T]\times M\to[0,+\infty]$ such that $\eta(t,\cdot)$ is inf-compact for all $t\in [0,T]$ and
    \[
    \sup_{\lambda\in\frakJ}\,\int_0^T\!\!\int_M\eta(t,u)\,\lambda(du,dt)< +\infty.
    \]
\end{enumerate}
\end{theorem}
\begin{proof}
See e.g. \cite[Definition 3.3]{balder00}
\end{proof}
\begin{theorem}[\textbf{Prohorov criterion for relative compactness}]\label{prohym}
Let $M$ be a metrisable Suslin space. Then every flexibly tight subset of $\Y(0,T;M)$ is sequentially relatively compact in the stable topology.
\end{theorem}
\begin{proof}
See \cite[Theorem 4.3.5]{cadfval}
\end{proof}

\begin{lem}\label{Lem:lscym}
Let $M$ be a metrisable Suslin space and let
\[
h:[0,T]\times M\to [-\infty,+\infty]
\]
be a measurable function such that $h(t,\cdot)$ is lower semi-continuous for every $t\in [0,T]$ and satisfies one of the two following conditions
\begin{enumerate}
  \item $\abs{h(t,u)}\le \gamma (t),$ a.e. $t\in [0,T],$ for some $\gamma\in L^1(0,T;\R),$
  \item $h\geq 0.$
\end{enumerate}
If $\lambda_n\to\lambda$ stably in $\Y(0,T;M),$ then
\[
\int_0^T\!\int_M h(t,u)\,\lambda(du,dt)\le \liminf_{n\to\infty}\int_0^T\!\int_M h(t,u)\,\lambda_n(du,dt).
\]
\end{lem}
\begin{proof}
If (1) holds, the result follows from Theorem 2.1.3--Part G in \cite{cadfval}. If (2) holds, the result follows from Proposition 2.1.12--Part (d) in \cite{cadfval}.
\end{proof}
%we suspect that the above result also holds if M is Radon and completely regular (or at least separable and metrisable)
The last two results will play an essential role in Section \ref{proofmain} in the proof of existence of stochastic optimal relaxed controls. They are, in fact, the main reasons why it suffices for our purposes to require that the control set $M$ is only metrisable and Suslin, in contrast with the existing literature on stochastic relaxed controls. Indeed, Theorem \ref{prohym} will be used to prove tightness of the laws of random Young measures (see Lemma \ref{tightrym} below) and Lemma \ref{Lem:lscym} will be used to prove the lower semi-continuity of the relaxed cost functionals as well as Theorem \ref{Lem:contym} below which will also be crucial to pass to the limit in the proof of our main result.

\begin{theorem}\label{Lem:contym}
Let $M$ be a metrisable Suslin space. If $\lambda_n\to\lambda$ stably in $\Y(0,T;M),$ then for every $f\in L^1(0,T;\mathcal{C}_b(M))$ we have
\[
\lim_{n\to\infty}\int_0^T\!\int_M f(t,u)\,\lambda_n(du,dt)=\int_0^T\!\int_M f(t,u)\,\lambda(dt,du).
\]
\end{theorem}
\begin{proof}
%\cite[Theorem 2.1.3, Part B]{cadfval})
Use Lemma \ref{Lem:lscym} with $f$ and $-f.$
\begin{comment}
By Lemma \ref{Lem:lscym} we have
\[
\int_0^T\!\int_M f(t,u)\,\lambda(du,dt)\le\liminf_{n\to\infty}\int_0^T\!\int_M f(t,u)\,\lambda_n(du,dt)
\]
The same argument applied to $-f$ gives
\[
\int_0^T\!\int_M -f(t,u)\,\lambda(du,dt)\le\liminf_{n\to\infty}\int_0^T\!\int_M -f(t,u)\,\lambda_n(du,dt)
\]
that is,
\[
\limsup_{n\to\infty}\int_0^T\!\int_M f(t,u)\,\lambda_n(du,dt)\leq\int_0^T\!\int_M f(t,u)\,\lambda(du,dt)
\]
and the desired result follows.
\end{comment}
\end{proof}
%if $M$ is metrisable and Suslin then we can take functions in $L^1(0,T;\mathcal{C}_b(M))$ in the definition of the stable topology
We will need the following version of the so-called Fiber Product Lemma. For a measurable map $y:[0,T]\to M,$ we denote by $\underline{\delta}_{y(\cdot)}(\cdot)$ the \emph{degenerate Young measure} defined as $\underline{\delta}_{y(\cdot)}(du,dt):=\delta_{y(t)}(du)\,dt.$
\begin{lem}[{\bf Fiber Product Lemma}]\label{Lem:fibpro}
Let $\mathcal{S}$ and $M$ be separable metric spaces and let $y_n:[0,T]\to\mathcal{S}$ be a sequence of measurable mappings which converge pointwise to a mapping $y:[0,T]\to \mathcal{S}.$ Let $\lambda_n\to\lambda$ stably in $\Y(0,T;M)$ and consider the following sequence of Young measures on $\mathcal{S}\times M,$
\[
(\underline{\delta}_{y_n}\otimes\lambda_n)(dx,du,dt):=\delta_{y_n(t)}(dx)\,\lambda_n(du,dt), \ \ n\in\mathds{N},
\]
and
\[
(\underline{\delta}_{y}\otimes\lambda)(dx,du,dt):=\delta_{y(t)}(dx)\,\lambda(du,dt).
\]
Then $\underline{\delta}_{y_n}\otimes\lambda_n\to \underline{\delta}_{y}\otimes\lambda$ stably in $\Y(0,T;S\times M).$
\end{lem}
\begin{proof}
%Proposition Corollary 3.1.5 in \cite{cadfval} implies that $\underline{\delta}_{y_n}\to \underline{\delta}_{y}$ stably in $\Y(0,T;\bE)$
Proposition 1 in \cite{val93} implies that $\underline{\delta}_{y_n}\to \underline{\delta}_{y}$ stably in $\Y(0,T;\mathcal{S}),$ and the result follows from Corollary 2.2.2 and Theorem 2.3.1 in \cite{casdefitte}.
\end{proof}
\begin{lem}\label{tightrym}
Assume $M$ is metrisable and Suslin. For each $n\in\mathds{N}$ let $\lambda_n$ be a random Young measure on $M$ defined on a probability space $(\Omega^n,\Fil^n,\Prob^n).$ Assume there exists a
measurable function $\eta:[0,T]\times M\to[0,+\infty]$ such that $\eta(t,\cdot)$ is inf-compact for all $t\in [0,T]$ and
\[
\Exp^{\Prob^n}\int_0^T\!\int_M\eta(t,u)\,\lambda_n(du,dt)\le R, \ \ \mbox{ for all } \ n\in\mathds{N}.
\]
for some $R>0.$ Then, the family of laws of $\{\lambda_n\}_{n\in\mathds{N}}$ is tight on $\Y(0,T;M).$
\end{lem}
\begin{proof}
For each $\varepsilon>0$ define the set
\[
K_\varepsilon:=\left\{\lambda\in\Y:\int_0^T\!\int_M\eta(t,u)\,\lambda(du,dt)\le\frac{R}{\varepsilon}\right\}.
\]
By Theorems \ref{flexiblytight} and \ref{prohym}, $K_\varepsilon$ is relatively compact in the stable topology of $\Y(0,T;M),$ and by Chebyshev's inequality we have
\[
\Prob^n\left(\lambda_n\in\Y\setminus\bar{K_\varepsilon}\right)\le \Prob^n\left(\lambda_n\in\Y\setminus K_\varepsilon\right)
\le \frac{\varepsilon}{R}\, \Exp^{\Prob^n}\int_0^T\!\!\int_M\eta(t,u)\,\lambda_n(du,dt)\le\varepsilon
\]
and the tightness of the laws of $\{\lambda_n\}_{n\geq1}$ follows.
\end{proof}
%
%the idea behind this proof is essentially the fact that the integral of the inf-compact function against the young measure is again inf-compact as a %function of the young measure. This is essentially the same idea behind the proof of tightness of the laws of the trajectories of the approximating %sequences

\subsection{Stochastic convolutions in UMD type 2 Banach spaces}
This section builds on the results on the factorization method for stochastic convolutions in UMD type 2 Banach spaces from \cite{brz1} and \cite{brzgat}. First, we recall the definition of the factorization operator as the negative fractional power of a certain abstract parabolic operator as well as some of its regularizing and compactness properties. Then, we review some basic properties of stochastic convolutions in M-type 2 Banach spaces.

In the sequel, $(\bE,\abs{\cdot}_\bE)$ will denote a Banach space and $T\in(0,+\infty)$ will be fixed. We start off by introducing the following Sobolev-type spaces,
\[
W^{1,p}(0,T;\bE):=\Bigl\{y\in L^p(0,T;\bE):y^\prime=\frac{dy}{dt}\in L^p(0,T;\bE)\Bigr\}, \ \ \ p>1
\]
where $y'$ denotes the weak derivative, and
\[
W_0^{1,p}(0,T;\bE):=\{y\in W^{1,p}(0,T;\bE): y(0)=0\}.
\]
Observe that $y(0)$ is well defined for $y\in W^{1,p}(0,T;\bE)$ since by the Sobolev Embedding Theorem we have $W^{1,p}(0,T;\bE)\subset \mathcal{C}([0,T];\bE),$ see e.g. \cite[Lemma 3.1.1]{temam1}.

Let $A$ be a closed linear operator on $\bE$ and let $D(A),$ the domain of $A,$ be endowed with the graph norm. We define the abstract parabolic operator $\Lambda_T$ on $L^p(0,T;\bE)$ through the formula
\begin{equation}\label{AT}
\begin{split}
D(\Lambda_T) &:= W^{1,p}_0(0,T;\bE)\cap L^p(0,T;D(A)),\\
\Lambda_Ty&:= y'+A(y(\cdot)).
\end{split}
\end{equation}
%We consider the weak derivative as closed operator on $L^p(0,T;\bE)$
Our aim is to define the factorization operator as the negative fractional powers of $\Lambda_T.$ This definition relies on the closedness of the operator $\Lambda_T,$ which will follow from the Dore-Venni Theorem, see \cite{DV}. This, however, requires further conditions on the Banach space $\bE$ and the operator $A.$
\begin{defi}
A Banach space $\bE$ is said to have the property of \emph{unconditional martingale differences} (and we say that $\bE$ is a \emph{UMD} space) iff for some $p\in (1,\infty)$ there exists a constant $c\geq0$ such that
\[
\Bigl|\!\Bigl|\sum_{k=0}^n\varepsilon_k(y_k-y_{k-1})\Bigr|\!\Bigr|_{L^p(\Omega,\Fil,\Prob;\bE)}
\le c\Bigl|\!\Bigl|\sum_{k=0}^n(y_k-y_{k-1})\Bigr|\!\Bigr|_{L^p(\Omega,\Fil,\Prob;\bE)}
\]
for all $n\in\mathds{N},$ $\varepsilon_k\in\{\pm 1\}$ and all $\bE-$valued discrete martingales $\{y_k\}_{k}$ with $y_{-1}=0.$
\end{defi}
%The above definition is $p-$independent (reference?)
%
\begin{rem}
A normed vector space $\bE$ is said to be $\zeta-$\emph{convex} iff there exists a symmetric, biconvex (i.e. convex in each component) function $\zeta:\bE^2\to\R$ such that $\zeta(0,0)>0$ and $\zeta(x,y)\le \abs{x+y}_\bE$ for any $x,y\in\bE$ with $\abs{x}_\bE=\abs{y}_\bE=1.$ Burkholder proved in \cite{burk1} that a Banach space is UMD iff it is $\zeta-$convex. Moreover, a necessary (see \cite{burk2}) and sufficient (see \cite{bour83}) condition for a Banach space $\bE$ to be UMD is that the Hilbert transform is bounded on $L^p(\R;\bE)$ for some $p\in(1,\infty).$
\end{rem}

\begin{ex}\label{exUMD}
Hilbert spaces and the Lebesgue spaces $L^{p}(\mathcal{O}),$ with $\mathcal{O}$ a bounded domain in $\R^d$ and $p\in (1,+\infty),$ are examples of UMD spaces, see e.g. \cite[Theorem 4.5.2]{amann}.
\end{ex}

\begin{defi}
Let $A$ be a linear operator on a Banach space $\bE.$ We say that $A$ is \emph{positive} if it is closed, densely defined, $(-\infty,0]\subset \rho(A)$ and there exists $C\geq 1$ such that
\[
\Vertt (\lambda I+A)^{-1}\Vertt_{\Lin(\bE)} \le  {C\over 1+\lambda}, \ \ \ \text{ for all }\lambda\geq 0.
\]
\end{defi}
\begin{rem}
It is well known that if $A$ is a positive operator on $\bE,$ then $A$ admits (not necessarily bounded) fractional powers $A^z$ of any order $z\in\mathds{C}$  (see e.g. \cite[Section 4.6]{amann}). Recall that, for $\abs{\Re z}\le 1,$ the fractional power $A^z$ is defined as the closure of the linear mapping
\begin{equation}\label{Az}
D(A)\ni x\mapsto \frac{\sin \pi z}{\pi z}\int_0^{+\infty} t^z(tI+A)^{-2}Ax\,dt\in\bE,
\end{equation}
see e.g. \cite[p. 153]{amann}.
%Moreover, if $\Re z\in (0,1),$ then $A^{-z}\in\Lin(\bE)$ and we have
%\[
%A^{-z}x=\frac{\sin \pi z}{\pi}\int_0^{+\infty}t^{-z}(tI+A)^{-1}x\,dt.
%\]
\end{rem}
\begin{defi}
The class $\BIP(\theta,\bE)$ of operators with \emph{bounded imaginary powers} on $\bE$ with parameter $\theta\in[0,\pi)$ is defined as the class of positive operators $A$ on $\bE$ with the property that $A^{is}\in \Lin(\bE)$ for all $s\in\R$ and there exists a constant $K>0$ such that
\begin{equation}
\Vertt A^{is} \Vertt_{\Lin(\bE)} \le K e^{\theta |s|}, \; s \in \R.
\label{2.1}
\end{equation}
\end{defi}
We will also denote $\BIP^-(\theta,\bE):=\cup_{\sigma\in(0,\theta)}\BIP(\sigma,\bE).$ Our main assumption on the operator $A$ throughout this article is the following,
\begin{Assumption}\label{Assum1}
$A\in\BIP^-({\pi \over 2},\bE).$
\end{Assumption}
In \cite[Theorem 2]{pruesohr1} it was proved that if $A$ satisfies Assumption \ref{Assum1}, then the operator $-A$ generates a (uniformly bounded) analytic $C_0-$semigroup $(S_t)_{t\geq 0}$ on $\bE.$ If furthermore $\bE$ is a UMD space, by the Dore-Venni Theorem (see Theorems 2.1 and 3.2 in \cite{DV}) it follows that the parabolic operator $\Lambda_T$ is positive on $L^p(0,T;\bE)$ and, in particular, admits the negative fractional powers $\Lambda_T^{-\alpha}$ for $\alpha\in(0,1].$ We have in fact the following formula

\begin{prop}[\cite{brz1}, Theorem 3.1]\label{Prop:2.1}
Let $\bE$ be a UMD Banach space and let Assumption \ref{Assum1} be satisfied. Then, for any $\alpha\in(0,1]$, $\Lambda_T^{-\alpha}$ is a bounded linear operator on $L^p(0,T;\bE)$, and  for $\alpha\in(0,1]$ we have
\begin{equation}
\left(\Lambda_T^{-\alpha}f\right)(t)=
{1 \over \Gamma (\alpha)}  \int_0^t (t-r)^{\alpha -1} S_{t-r}f(r) \,
dr, \;\; t \in (0,T),\; \;  f \in L^p(0,T;\bE).
\label{2.13a}
\end{equation}
\end{prop}
The fractional powers $\Lambda_T^{-\alpha}$ also satisfy the following compactness property which will be crucial to infer tightness of a certain family of laws of processes in the proof of our main Theorem,
\begin{theorem}[\cite{brzgat}, Theorem 2.6]\label{Th:compact}
Under the same assumptions of Proposition \ref{Prop:2.1}, suppose further that $A^{-1}$ is a compact operator (i.e. the embedding  $D(A)\imbed\bE$ is compact). Then, for any $ \alpha \in (0,1]$, the  operator $\Lambda_T^{-\alpha}$ is compact on $L^p(0,T;\bE).$
\end{theorem}
The following smoothing property of $\Lambda_T^{-\alpha}$ is a particular case of a more general regularizing result (see Lemma 3.3 in \cite{brz1}).
\begin{lem}\label{L:reg}
Under the same assumptions of Proposition \ref{Prop:2.1}, let $\alpha$ and  $\delta$ be positive numbers satisfying
\begin{equation}
\delta+\frac{1}{p}<\alpha
\label{cond:1}
\end{equation}
Then $\Lambda_{T}^{-\alpha}f\in\mathcal{C}([0,T];D(A^\delta))$ for all $f	\in	 L^p(0,T;\bE)$ and $\Lambda_{T}^{-\alpha}$ is a bounded operator from $L^p(0,T;\bE)$ into $\mathcal{C}([0,T];D(A^\delta)).$
\end{lem}
Using Theorem \ref{Th:compact} and Lemma \ref{L:reg} one can prove the following,
\begin{cor}[\cite{brzgat}, Corollary 2.8]\label{C:comp}
Suppose the assumptions of Theorem \ref{Th:compact} and Lemma \ref{L:reg} are satisfied. Then $\Lambda_T^{-\alpha}$  is a compact map from $L^p(0,T;\bE)$  into $\mathcal{C}([0,T];D(A^\delta))$.
\end{cor}
\begin{rem}\label{nu0}
Since $T>0$ is finite, it can be proved that the above results are still valid if $A+\nu I\in\BIP^-(\frac{\pi}{2},\bE)$ for some $\nu\geq 0$ (see e.g. \cite{brzgat} or \cite[Theorem 4.10.8]{amann}).
\end{rem}

\begin{ex}\label{ex0}
Let $\mathcal{O}$ be a bounded domain in $\R^d$ with sufficiently smooth boundary and let $\mathcal{A}$ denote the second-order elliptic differential operator defined as
\[
(\mathcal{A}x)(\xi):=-\sum_{i,j=1}^d a_{ij}(\xi)\frac{\partial^2 x}{\partial \xi_i\partial \xi_j}+\sum_{i=1}^d b_{i}(\xi)\frac{\partial x}{\partial \xi_i} + c(\xi)x(\xi), \ \ \ \xi\in\mathcal{O},
\]
with $a_{ij}=a_{ji},$
\[
\sum_{i,j=1}^d a_{ij}(\xi)\lambda_i\lambda_j\geq C\abs{\lambda}^2, \ \ \lambda\in\R^d.
\]
and $c,b_i,a_{ij}\in\mathcal{C}^\infty(\bar{\mathcal{O}}).$ Let $q\geq 2$ and let $A_q$ denote the realization of $\mathcal{A}$ in $L^q(\mathcal{O}),$ that is,
\begin{equation}\label{Aqdef}
\begin{split}
D(A_q)&:=H^{2,q}(\mathcal{O})\cap H_0^{1,q}(\mathcal{O}),\\
A_qx&:=\mathcal{A}y.
\end{split}
\end{equation}
Then $A_q+\nu I\in\BIP^-({\pi \over 2},L^q(\mathcal{O}))$ for some $\nu\geq 0 $ (see e.g. \cite{seeley}).
Other examples of differential operators satisfying such condition include realizations of higher order elliptic partial differential operators \cite{seeley} and the Stokes operator \cite{gigasohr}.
\end{ex}
We now briefly outline the construction of the stochastic integral in M-type 2 Banach spaces with respect to a cylindrical Wiener process. For the details we refer to \cite{dett} or \cite{brz1}. See also \cite{brz2} and the references therein.

For the rest of this section we fix a separable Hilbert space $\left(\bH,[\cdot,\cdot]_\bH\right)$ and a probability space $(\Omega ,\Fil,\Prob)$ endowed with a filtration $\mathds{F}=\{\Fil_t\}_{t\ge 0}.$
\begin{defi}
A family $W(\cdot)=\{W(t)\}_{t\geq 0}$ of bounded linear operators from $\bH$ into $L^2(\Omega;\R)$ is called a $\bH-$\emph{cylindrical Wiener process} (with respect to the filtration $\mathds{F})$ iff
\begin{enumerate}[(i)]
\item $\Exp\, W(t)y_1 W(t)y_2 = t[y_1 ,y_2]_\bH$, for all $t\ge 0$ and $y_1 ,y_2 \in\bH$,

\smallskip
\item for each $y\in\bH$, the process $\{W(t)y\}_{t\geq 0}$ is a standard one-dimensional Wiener process with respect to $\mathds{F}.$
\end{enumerate}
\end{defi}
For $h\in\bH$ and $x\in\bE,$ $h\otimes x$ will denote the linear operator
\[
(h\otimes x)y:=[h,y]_\bH x, \ \ \ y\in\bH.
\]
For $p\geq 1$ and a Banach space $(V,\abs{\cdot}_V),$ let $\M^p(0,T;V)$ denote the space of (classes of equivalences of) $\mathds{F}-$progressively measurable processes $\Phi:[0,T]\times\Omega\to V$ such that
\[
\norm{\Phi}^p_{\M^p(0,T;V)}:=\Exp\int_0^T\abs{\Phi(t)}^p_V\,dt<\infty.
\]
Notice that $\M^p(0,T;V)$ is a Banach space with the norm $\norm{\cdot}_{\M^p(0,T;V)}.$

\begin{defi}
A process $\Phi(\cdot)$ with values in $\Lin(\bH,\bE)$ is said to be \emph{elementary} (with respect to the filtration $ \{\Fil_t\}_{t\ge 0})$ if there exists a partition $0=t_0<t_1\cdots<t_N=T$ of $[0,T]$ such that
\[
\Phi(s)=\sum_{n=0}^{N-1}\sum_{k=1}^K \mathbf{1}_{[t_n,t_{n+1})}(s)e_k\otimes\xi_{kn}, \ \ \ s\in [0,T].
\]
where $(e_k)_k$ is an orthonormal basis of $\bH$ and $\xi_{kn}$ is a $\bE-$valued $\Fil_{t_n}-$measurable random variable, for $n=0,1,\dots,N-1.$
For such processes we define the \emph{stochastic integral} as
\[
I_T(\Phi):=\int_0^T\Phi(s)\,dW(s):=\sum_{n=0}^{N-1}\sum_{k=1}^K\left(W(t_{n+1})e_k-W(t_{n})e_k\right)\xi_{kn}.
\]
\end{defi}
\begin{defi}
Let $(\gamma_k)_k$ be a sequence of real-valued standard Gaussian random variables. A bounded linear operator $R:\bH\to\bE$ is said to be $\gamma-$\emph{radonifying} iff there exists an orthonormal basis $(e_k)_{k\geq 1}$ of $\bH$ such that the sum $\sum_{k\geq 1}\gamma_k Re_k$ converges in $L^2(\Omega;\bE).$
\end{defi}
We denote by $\gamma(\bH,\bE)$ the class of $\gamma-$radonifying operators from $\bH$ into $\bE,$ which is a Banach space equipped with the norm
\[
\norm{R}^2_{\gamma(\bH,\bE)}:=\Exp\Bigl|\sum_{k\geq 1}\gamma_k Re_k\Bigr|^2_\bE , \ \ \ \ R\in \gamma(\bH,\bE).
\]
The above definition is independent of the choice of the orthonormal basis $(e_k)_{k\geq 1}$ of $\bH.$ Moreover, $\gamma(\bH,\bE)$ is continuously embedded into $\Lin(\bH,\bE)$ and is an operator ideal in the sense that if $\bH'$ and $\bE'$ are Hilbert and Banach spaces respectively such that $S_1\in\Lin(\bH',\bH)$ and $S_2\in\Lin(\bE,\bE')$ then $R\in \gamma(\bH,\bE)$ implies $S_2RS_1\in\gamma(\bH',\bE')$ with
\[
\norm{S_2RS_1}_{\gamma(\bH',\bE')}\le \norm{S_2}_{\Lin(\bE,\bE')}\norm{R}_{\gamma(\bH,\bE)}\norm{S_1}_{\Lin(\bH',\bH)}
\]
It can also be proved that $R\in \gamma(\bH,\bE)$ iff $RR^*$ is the covariance operator of a centered Gaussian measure on $\B(\bE),$ and if $\bE$ is a Hilbert space, then $\gamma(\bH,\bE)$ coincides with the space of Hilbert-Schmidt operators from $\bH$ into $\bE$ (see e.g. \cite{vn0} and the references therein). There is also a very useful characterization of $\gamma-$radonifying operators if $\bE$ is a $L^p-$space,
\begin{lem}[\cite{vnvw}, Lemma 2.1]\label{gammalp}
Let $(S,\mathfrak{A},\rho)$ be a $\sigma-$finite measure space and let $p\in [1,\infty).$ Then, for an operator $R\in\Lin(\bH,L^p(S))$ the following assertions are equivalent
\begin{enumerate}
  \item $R\in\gamma(\bH,L^p(S)),$
  \item There exists a function $g\in L^p(S)$ such that for all $y\in\bH$ we have
  \[
  \abs{(Ry)(s)}\le \abs{y}_\bH\cdot g(s), \ \ \ \rho-\mbox{a.e.} \ s\in S.
  \]
\end{enumerate}
In such situation, there exists a constant $c>0$ such that $\norm{R}_{\gamma(\bH,L^p(S))}\le c\abs{g}_{L^p(S)}.$
\end{lem}

\begin{defi}
$\phantom{banach}$
\begin{enumerate}
\item  A Banach space $\bE$ is said to be of \emph{martingale type} $2$ (and we write $\bE$ is \emph{M-type} $2$) iff there exists a constant $C_2>0$ such that
\begin{equation}\label{mtype2}
\sup_{n} \mathbb{E} | M_{n} |_\bE ^{2} \le C_2 \sum_{n} \mathbb{E}  | M_{n}-M_{n-1} |_\bE ^{2}
\end{equation}
for any $\bE-$valued discrete martingale $\{M_{n}\}_{n\in \mathds{N}}$ with $M_{-1}=0.$

\smallskip\item $\bE$ is said to be of \emph{type} $2$ iff there exists $K_2>0$ such that
\begin{equation}
 \Exp\Bigl| \sum_{i=1}^n \epsilon_i x_i\Bigr|_\bE^2 \le K_2\sum_{i=1}^n | x_i |_\bE^2
\end{equation}
for any finite sequence $\{\epsilon_i\}_{i=1}^n$ of $\{-1,1\}-$valued symmetric i.i.d. random variables and for any finite sequence $\{x_i\}_{i=1}^n$ of elements of $\bE.$
\end{enumerate}
\end{defi}
\begin{prop}[\cite{brz1}, Proposition 2.11]
Let $\bE$ be a UMD and type 2 Banach space. Then $\bE$ is M-type 2.
\end{prop}

\begin{ex}
Let $\mathcal{O}$ be a bounded domain in $\R^d.$ Then the Lebesgue spaces $L^{p}(\mathcal{O})$ are both type 2 and M-type $2,$ for $p\in [2,\infty).$
\end{ex}

If $\bE$ is a M-type 2 Banach space, it is easy to show (see e.g. \cite{dett}) that the stochastic integral $I_T(\Phi)$ for elementary processes $\Phi(\cdot)$ satisfies
\begin{equation}\label{burk0}
\Exp\abs{I_T(\Phi)}_\bE^2\le C_2 \Exp\int_0^T\norm{\Phi(s)}_{\gamma(\bH,\bE)}^2\,ds
\end{equation}
where $C_2$ is the same constant in (\ref{mtype2}). Since the set of elementary processes is dense in $\M^2(0,T;\gamma(\bH,\bE))$ (see e.g. \cite[Ch. 2, Lemma 18]{neidhardt}) by (\ref{burk0}) the linear mapping $I_T$ extends to a bounded linear operator from $\M^2(0,T;\gamma(\bH,\bE))$ into $L^2(\Omega;\bE).$ We denote this operator also by $I_T.$

Finally, for each $t\in [0,T]$ and $\Phi\in \M^2(0,T;\gamma(\bH,\bE)),$ we define
\[
\int_0^t\Phi(s)\,dW(s):=I_T(\mathbf{1}_{[0,t)}\Phi).
\]
The process $\int_0^t\Phi(s)\,dW(s), \ t\in [0,T],$ is a martingale with respect to $\mathds{F}.$ Moreover, we have the following Burkholder Inequality
\begin{prop}
Let $\bE$ be a M-type 2 Banach space. Then, for any $p\in (0,+\infty)$ there exists a constant $C=C(p,\bE)$ such that for all $\Phi\in\M^2(0,T;\gamma(\bH,\bE))$ we have
\[
\Exp\biggl[\sup_{\, t\in [0,T]}\Bigl|\int_0^t\Phi(s)\,dW(s)\Bigr|_\bE^p\biggr]\le{\textstyle{\left(\frac{p}{p-1}\right)^p}} C\cdot\Exp\left[\left(\int_0^T\norm{\Phi(s)}_{\gamma(\bH,\bE)}^2\,ds\right)^{p/2}\right]
\]
\end{prop}
\begin{proof}
See Theorem 2.4 in \cite{brz1}.
\end{proof}
Let $M(\cdot)$ be a $\bE-$valued continuous martingale with respect to the filtration $\mathds{F}=\{\Fil_t\}_{t\ge 0}$ and let $\seq{\cdot,\cdot}$ denote the duality between $\bE$ and $\bE^*.$ The \emph{cylindrical quadratic variation} of $M(\cdot),$ denoted by $[M],$ is defined as the (unique) cylindrical process (or linear random function) with values in $\Lin(\bE^*,\bE)$ that is $\mathds{F}-$adapted, increasing and satisfies
\begin{enumerate}
\item $[M](0)=0$
\item for arbitrary $x^*,y^*\in\bE^*,$ the real-valued process
\[
\seq{M(t),x^*}\seq{M(t),y^*}-\seq{[M](t)x^*,y^*}, \ \ \ t\geq 0
\]
is a martingale with respect to $\mathds{F}.$
\end{enumerate}
%In other words, $\{[M](t)\}_{t\geq 0}$ is the quadratic variation of the one-dimensional martingale $\set{M(t)\otimes M(t)}_{t\geq 0}.$
For more details on this definition we refer to \cite{dett}. We will need the following version of the Martingale Representation Theorem in M-type 2 Banach spaces, see Theorem 2.4 in \cite{dett} (see also \cite{ondre2}),
\begin{theorem}\label{martrep}
Let $(\Omega,\Fil,\mathds{F},\Prob)$ be a filtered probability space and let $\bE$ be a separable M-type 2 Banach space. Let $M(\cdot)$ be a $\bE-$valued continuous square integrable $\mathds{F}-$martingale with cylindrical quadratic variation process of the form
\[
[M](t)=\int_0^t g(s)\circ g(s)^*\,ds, \ \ \ \ t\in [0,T],
\]
where $g\in\M^2(0,T;\gamma(\bH,\bE)).$ Then, there exists a probability space $(\tOmega,\tFil,\tProb),$ extension of $(\Omega,\Fil,\Prob),$ and a $\bH-$cylindrical Wiener process $\{\tilde{W}(t)\}_{t\ge 0}$ defined on $(\tOmega,\tFil,\tProb),$ such that
\[
M(t)=\int_0^t g(s)\,d\tilde {W}(s), \ \ \ \tProb-\text{a.s.}, \ \ \ t\in [0,T].
\]
\end{theorem}
%
\begin{comment}
\begin{rem}
A Banach space $\bE$ is said to be $p-$smoothable (or $p-$uniformly smooth) iff $\bE$ can be equivalently renormed in such a way that  its modulus of smoothness
\[
\rho(t) = \sup  \Bigl\{ {1\over 2} \left( | x+ty |_\bE + | x-ty |_\bE \right)-1
:  | x |_\bE , | y |_\bE =1\Bigr\}
\]
satisfies $\rho(t) \le  K t^{p}$ for all $t>0$ and some $K>0$. It is well known, see \cit{P1}, \cit{P2}, that $\bE$ is
of M-type $p$ iff it is $p-$smoothable.
\end{rem}
\end{comment}
%We fix $w\in\R$ such that the semigroup generated by $-A-w$ is uniformly exponentially stable
Finally, we recall some aspects of the factorization method for stochastic convolutions in UMD type-2 Banach spaces. Recall that, under Assumption \ref{Assum1}, the operator $-A$ generates an analytic $C_0-$semigroup $(S_t)_{t\geq 0}$ on $\bE.$
\begin{lem}[\cite{brzgat}, Lemma 3.7]\label{Lem:A3}
Let $\bE$ be a UMD type-2 Banach space and let Assumption \ref{Assum1} be satisfied. Let $p\geq 2, \ \sigma\in[0,\frac{1}{2})$ and
$g(\cdot)$ be a $\Lin(\bH,\bE)-$valued stochastic process satisfying
\begin{equation}\label{sigma-eta}
A^{-\sigma} g(\cdot) \in \mathcal{M}^p\left(0,T;\gamma(\bH,\bE)\right).
\end{equation}
Let $\alpha>0$ be such that $\alpha+\sigma<\frac{1}{2}.$ Then, for every $t\in [0,T]$ the stochastic integral
\begin{equation}
y(t):=
{1 \over \Gamma(1 - \alpha)}
\int_0^t (t-r)^{-\alpha} S_{t-r}g(r)\, dW(r),
\label{2.17}
\end{equation}
exists and the process $y(\cdot)$ satisfies
\begin{equation}
%\Exp\int_0^T\abs{y(t)}^p_\bE\,dt
|\!|y|\!|_{\mathcal{M}^p\left(0,T;\bE\right)}
\le C\,T^{{1\over 2}-\alpha-\sigma}
|\!| A^{-\sigma }g|\!|_{\mathcal{M}^p\left(0,T;\gamma(\bH,\bE)\right)}
%\Exp\int_0^T\Vertt A^{-\sigma}g(t)\Vertt^p_{\gamma(\bH;\bE)}\,dt
\label{equ:A3}
\end{equation}
for some constant $C=C(\alpha,p,A,\bE),$ independent of $g(\cdot)$ and $T.$ In particular, the process $y(\cdot)$ has trajectories in $L^p(0,T;\bE),$ $\Prob-$a.s.
\end{lem}
\begin{theorem}[\cite{brz1}, Theorem 3.2]\label{Th:2.1}
Under the same assumptions of Lemma \ref{Lem:A3}, the stochastic convolution
\begin{equation}\label{2.19}
v(t)=\int^t_0 S_{t-r}g (r)\,dW(r), \ \  t\in [0,T],
\end{equation}
is well-defined and there exists a modification $\tv(\cdot)$ of $v(\cdot)$ such that ${\tilde v}(\cdot)\in D(\Lambda_T^\alpha),$ $\Prob-$a.s.
and the following `factorization formula' holds
\begin{equation}
\tv(t)=(\Lambda_T^{-\alpha}y)(t), \ \ \ \Prob-\text{a.s.}, \ \ t\in [0,T],
\label{2.splot}
\end{equation}
where $y(\cdot)$ is the process defined in \rf{2.17}. Moreover,
\[
\Exp|\!|{\tv(\cdot)}|\!|^p_{D(\Lambda_T^\alpha)} \le
C^pT^{p({1 \over 2} -\alpha-\sigma)}|\!| A^{-\sigma }g|\!|_{\mathcal{M}^p\left(0,T;\gamma(\bH,\bE)\right)}.
\]
\end{theorem}
\begin{cor}\label{Co:2}
Under the same assumptions of Theorem \ref{Th:2.1}, let $\delta$ satisfy
\begin{equation}
\delta + \sigma +\frac{1}{p}< { 1\over 2}.
\label{cond:2}
\end{equation}
Then, there exists a stochastic process $\tv(\cdot)$ satisfying
\begin{equation}
{\tilde v}(t)=\int^t_0 S_{t-r}g (r)\,dW(r), \ \  \Prob-\text{a.s.}, \ \
t \in [0,T],
\end{equation}
such that ${\tilde v}(\cdot)\in \mathcal{C}([0,T];D(A^\delta)),$ $\Prob-$a.s. and
\[
\Exp|\!|{\tilde v(\cdot)}|\!|^p_{\mathcal{C}([0,T];D(A^\delta))}
%\Exp\biggl[\sup_{\, t\in [0,T]}\abs{{\tilde v}(t)}_\delta^p\biggr]
\le C_T
%%T^{({1 \over 2} -\alpha)p}
|\!| A^{-\sigma }g|\!|^p_{\mathcal{M}^p\left(0,T;\gamma(\bH,\bE)\right)}
\]
\end{cor}
\begin{proof}
Follows from Theorem \ref{Th:2.1} and Lemma \ref{L:reg}, by taking $\alpha$ such that
\[
\delta +\frac{1}{p}<\alpha<\sigma- { 1\over 2}.
\]
\end{proof}
\begin{ex}\label{ex1}
Let $\bH$ be a separable Hilbert space and let $\mathcal{A}$ be the second order differential operator introduced in Example \ref{ex0}. Let $A_q$ be the realization of $\A+\nu I$ on $L^q(\mathcal{O}),$ where $\nu\geq 0$ is chosen such that $A_q\in\BIP^-(\frac{\pi}{2},L^q(\mathcal{O})).$ Fix $q>d$ and $\sigma$ satisfying
\begin{equation}\label{sigmadq}
\frac{d}{2q}<\sigma<\frac{1}{2}.
\end{equation}
Then, if $g\in\M^p\left(0,T;\Lin(\bH,L^q(\mathcal{O}))\right)$ we have
\begin{equation}\label{Aqsigmaeta}
A_q^{-\sigma}g\in\M^p\left(0,T;\gamma(\bH,L^q(\mathcal{O}))\right).
\end{equation}
Therefore, if $\alpha<\frac{1}{2}-\sigma,$ the statements in Lemma \ref{Lem:A3}, Theorem \ref{Th:2.1} and Corollary \ref{Co:2} apply.
\end{ex}
\begin{proof}[Proof of (\ref{Aqsigmaeta})]
From (\ref{sigmadq}) we have in particular $\sigma>1/2q.$ Hence, from \cite[Theorem 1.15.3]{triebel}, we have
\[
D(A_q^\sigma)=[L^q(\mathcal{O}),D(A_q)]_\sigma=H^{2\sigma,q}_0(\mathcal{O})
\]
isomorphically, and also by (\ref{sigmadq}), we have
\[
H^{2\sigma,q}_0(\mathcal{O})\hookrightarrow \mathcal{C}_0(\bar{\mathcal{O}}).
\]
Let $c_{\sigma,q}>0$ be such that $\abs{x}_{\mathcal{C}_0(\bar{\mathcal{O}})}\le c_{\sigma,q}\abs{x}_{H^{2\sigma,q}_0(\mathcal{O})}, \ x\in H^{2\sigma,q}_0(\mathcal{O}).$ Then, for any $y\in \bH$ and $t\in [0,T]$ we have
\begin{align*}
\abs{A^{-\sigma}g(t)y}_{L^\infty(\mathcal{O})}&\le c_{\sigma,q}\abs{A^{-\sigma}g(t)y}_{D(A_q^\sigma)}\\
&\le c_{\sigma,q}\left(1+|\!|A^{-\sigma }|\!|_{\Lin(L^q(\mathcal{O}))}\right)\abs{g(t)y}_{L^q(\mathcal{O})}\\
&\le c_{\sigma,q}\left(1+|\!|A^{-\sigma }|\!|_{\Lin(L^q(\mathcal{O}))}\right)\norm{g(t)}_{\Lin(\bH,L^q(\mathcal{O}))}\abs{y}_{\bH}.
\end{align*}
Hence, by Lemma \ref{gammalp}, there exists $c'>0$ such that
\[
\left|\!\left|A_q^{-\sigma}g(t)\right|\!\right|_{\gamma(\bH,L^q(\mathcal{O}))}\le c'\norm{g(t)}_{\Lin(\bH,L^q(\mathcal{O}))}
\]
and (\ref{Aqsigmaeta}) follows.
\end{proof}
\begin{ex}\label{ex2}
Let $A_q$ be again as above and let $\bH=H^{\theta,2}(\mathcal{O})$ with $\theta>\frac{d}{2}-1.$ By Lemma 6.5 in \cite{brz1}, if $\sigma$ satisfies
\[
\sigma>\frac{d}{4}-\frac{\theta}{2}
\]
then $A_q^{-\sigma}$ extends to a bounded linear operator from $H^{\theta,2}(\mathcal{O})$ to $L^q(\mathcal{O}),$ which we again denote by $A_q^{-\sigma},$ such that
\[
A_q^{-\sigma}\in \gamma\left(H^{\theta,2}(\mathcal{O}),L^q(\mathcal{O})\right).
\]
Hence, by the right-ideal property of $\gamma-$radonifying operators, the statements in Lemma \ref{Lem:A3}, Theorem \ref{Th:2.1} and Corollary \ref{Co:2} hold true with the condition (\ref{sigma-eta}) now replaced by
\[
g\in\M^p\bigl(0,T;\Lin(H^{\theta,2}(\mathcal{O}))\bigr).
\]
\end{ex}
\begin{rem}
Observe in Example \ref{ex1} that, although we require $q>d,$ the choice of $\sigma$ is independent of the Hilbert space $\bH.$ In Example \ref{ex2}, however, the statement holds true for any value of $q$ but the choice of $\sigma$ depends on the Hilbert space $H^{\theta,2}(\mathcal{O}).$
\end{rem}
The stochastic convolution process defined by (\ref{2.19}) is frequently defined as the \emph{mild solution} to the stochastic Cauchy problem
\begin{equation}
\begin{split}
dX(t)+AX(t)&=g(t)\,dW(t),\\
X(0)&=0.
\end{split}
\end{equation}
The following result shows that, under some additional assumptions, the process $v(\cdot)$ is indeed a strict solution: let $D_A(\frac{1}{2},2)$ be the real interpolation space between $D(A)$ and $\bE$ with parameters $(\frac{1}{2},2),$ that is,
\[
{D}_{A}({\textstyle{1\over 2}},2):=\Bigl\{ x\in\bE:|x
|^{2}_{D_{A}(\frac{1}{2} ,2)} = \int^{1}_{0}| AS_tx|_\bE^{2}\,
dt <+\infty \Bigr\}.
%\label{eqn-4.2}
\]
\begin{lem}\label{Lem:strong}
Let $\bE$ be a UMD type-2 Banach space and let Assumption \ref{Assum1} be satisfied. Let $\xi\in L^2\left(\Omega, {\mathcal{F}}_0,\Prob;D_A(\frac{1}{2},2)\right)$, $g\in \mathcal{M}^2\left(0,T;\gamma(\bH,D_A(\frac{1}{2},2))\right)$ and $f\in \mathcal{M}^2(0,T;D(A^\zeta))$ for some $\zeta\geq 0.$ Then the following conditions are equivalent,
\begin{enumerate}[(i)]
\item $\displaystyle{X(t)= S_t\xi + \int^{t}_{0}S_{t-r}f(r)\, dr+ \int^{t}_{0}S_{t-r}g(r)\,dW(r)  , \ \Prob-\mbox{\rm  a.s., } \ t\in [0,T].}$

\smallskip\item The process $X(\cdot)$ belongs to $\M^2(0,T;D(A))\cap\mathcal{C}\left(0,T;L^2(\Omega,D_A(\frac{1}{2},2))\right)$ and satisfies
\[
X(t)+ \int^{t}_{0}AX(s)\, ds = \xi+ \int^{t}_{0}f(s) \, ds+\int^{t}_{0}g(s)\,dW(s) , \ \Prob-\mbox{\rm a.s.,  } \ t\in[0,T].
\]
\end{enumerate}
\end{lem}
\begin{proof}
See Proposition 4.2 in \cite{brz0}.
\end{proof}
\section[Existence of optimal relaxed controls in the weak formulation]{Existence of optimal relaxed controls\\ in the weak formulation}
Let $(B,\abs{\cdot}_B)$ be a Banach space continuously embedded into $\bE$ and let $M$ be a metrisable control set. We are concerned with a control system consisting of a cost functional of the form
\begin{equation}\label{cf}
J(X,u)=\Exp\left[\int_0^T h(s,X(s),u(s))\,ds+\varphi(X(T))\right].
\end{equation}
and a controlled semilinear stochastic evolution equation of the form
\begin{equation}\label{eq0}
\begin{split}
dX(t)+AX(t)\,dt&=F(t,X(t),u(t))\,dt+G(t,X(t))\,dW(t), \ \ \ t\in
[0,T]\\
X(0)&=x_0\in B
\end{split}
\end{equation}
where $W(\cdot)$ is a $\bH-$cylindrical Wiener process, $F:[0,T]\times B\times M\to B$ and, for each $(t,x)\in [0,T]\times B,$ $G(t,x)$ is a, possibly unbounded, linear mapping from $\bH$ into $\bE.$ More precise conditions on the coefficients $F$ and $G$ and on the functions $h$ and $\varphi$ are given below.

We associate the original control system (\ref{cf})-(\ref{eq0}) with a relaxed control system by extending the definitions of the nonlinear drift coefficient $F$ and the running cost function as follows: we define the \emph{relaxed coefficient} $\bar{F}$ through the formula
\begin{equation}\label{barf}
\bar{F}(t,x,\rho):=\int_M F(t,x,u)\,\rho(du), \ \ \ \rho\in\mathcal{P}(M), \ \ \ (t,x)\in [0,T]\times B,
\end{equation}
whenever the above expression is well-defined i.e. the map $M\ni u\mapsto F(t,x,u)\in B$ is Bochner integrable with respect to $\rho\in\mathcal{P}(M).$ Similarly, we define the \emph{relaxed running cost} function $\bar h.$ With these notations, the controlled equation (\ref{eq0}) in the relaxed control setting becomes, formally,
\begin{equation}\label{eq1}
\begin{split}
dX(t)+AX(t)\,dt&=\bar F(t,X(t),q_t)\,dt+G(t,X(t))\,dW(t), \ \ \ t\in
[0,T],\\
X(0)&=x_0\in B,
\end{split}
\end{equation}
where $\{q_t\}_{t\geq 0}$ is a $\mathcal{P}(M)$-valued relaxed control process, and the associated relaxed cost functional is defined as
\[
J(X,q)=\Exp\left[\int_0^T\bar h(s,X(s),q_s)\,ds+\varphi\left(X(T)\right)\right].
\]
%Recall that under Assumption \ref{Assum1}, the operator $-A$ generates an analytic semigroup $(S_t)_{t\geq 0}$ on $\bE.$
We will assume that the realization $A_B$ of the operator $A$ in $B,$
\begin{align*}
D(A_B)&:=\set{x\in D(A)\cap B:Ax\in B }\\
A_B&:=A|_{D(A_B)}
\end{align*}
is such that $-A_B$ generates a $C_0-$semigroup on $B,$ also denoted by $(S_t)_{t\geq 0}.$ We will only consider \emph{mild} solutions to equation (\ref{eq1}) i.e. solutions to the integral equation
\begin{equation}
X(t)=S_tx_0+\int^t_0 S_{t-r} \bar F(r,X(r),q_r)\,dr+\int^t_0 S_{t-r} G(r,X(r))\,dW(r), \ \ t\in [0,T].
\end{equation}
Our aim is to study the existence of optimal controls for the above stochastic relaxed control system in the following weak formulation,
\begin{defi}
Let $x_0\in B$ be fixed. A \emph{weak admissible relaxed control} is a system
\begin{equation}
\pi=\left(\Omega,\Fil,\mathds{P},\mathds{F},\{W(t)\}_{t\geq 0},\{q_t\}_{t\geq 0},\{X(t)\}_{t\geq 0}\right)
\lab{wcontrol-system}
\end{equation}
such that
\begin{enumerate}[{\rm (i)}]
    \item $\left(\Omega,{\Fil},\mathds{P}\right)$ is a complete probability space endowed with the filtration $\mathds{F}=\{\Fil_t\}_{t\geq 0}$

    \smallskip

    \item $\{W(t)\}_{t\ge 0}$ is a $\bH-$cylindrical Wiener process with respect to $\mathds{F}$

    \smallskip

    \item $\{q_t\}_{t\ge 0}$ is a $\mathds{F}-$adapted $\mathcal{P}(M)$-valued relaxed control process

    \smallskip

    \item $\{X(t)\}_{t\geq 0}$ is a $\mathds{F}-$adapted $B-$valued continuous process such that for all $t\in [0,T],$
    \begin{equation}\label{Xeq}
X(t)=S_tx_0+\int^t_0 S_{t-r} \bar F(r,X(r),q_r)\,dr+\int^t_0 S_{t-r} G(r,X(r))\,dW(r), \; \Prob-\text{a.s.}
\end{equation}
%This process is unique in the sense that if $\tX(\cdot)$ also satisfies equation (\ref{Xeq}) for all $t\in [0,T]$ then $\tX(\cdot)$ is a modification of $X(\cdot).$

    \smallskip

\item The mapping
\[
[0,T]\times \Omega\ni(t,\omega)\mapsto\bar h(t,X(t,\omega),q_t(\omega))\in\R
\]
belongs to $L^1([0,T]\times \Omega;\R)$ and $\varphi(X(T))\in L^1(\Omega;\R).$
\end{enumerate}
The set of weak admissible relaxed controls will be denoted by $\bar{\U}_{\rm ad}^{\rm w}(x_0).$ Under this weak formulation, the relaxed cost functional is defined as
\[
\bar{J}({\pi}):=\Exp^\Prob\left[\int_0^T\bar h(s,X^{\pi}(s),q_s^\pi)\,ds+\varphi\left(X^{\pi}(T)\right)\right], \ \ \ \pi\in\bar{\U}_{\rm ad}^{\rm w}(x_0),
\]
where $X^{\pi}(\cdot)$ is state-process corresponding to the weak admissible relaxed control $\pi.$ The relaxed control problem {\rm \textbf{(RCP)}} is to minimize $\bar{J}$ over $\bar{\U}_{\rm ad}^{\rm w}(x_0)$ for $x_0\in B$ fixed. Namely, we seek $\widetilde{\pi}\in\bar{\U}_{\rm ad}^{\rm w}(x_0)$ such that \[
\bar{J}(\widetilde{\pi})=\inf_{{\pi}\in\bar{\U}_{\rm ad}^{\rm w}(x_0)}\bar J({\pi}).
\]
\end{defi}
The following will be the main assumptions on the Banach space $B$ and the diffusion coefficient $G.$
\begin{Assumption}\label{Assum2}
There exist positive constants $\sigma$ and $\delta$ such that $\sigma+\delta<\frac{1}{2}$ and
\begin{enumerate}
\item $D(A^{\delta})\imbed B$

\item For all $(t,x)\in[0,T]\times B,$ $A^{-\sigma}G(t,x)$ extends to a bounded linear operator that satisfies
\[
A^{-\sigma}G(t,x)\in\gamma(\bH,\bE).
\]
Moreover, the mapping $[0,T]\times B\ni (t,x)\mapsto A^{-\sigma}G(t,x)\in\gamma(\bH,\bE)$ is bounded, continuous with respect to $x\in B$ and measurable with respect to $t\in [0,T].$
\end{enumerate}
\end{Assumption}
To formulate the main hypothesis on the drift coefficient $-A+F$ we need the notion of sub-differential of the norm. Recall that for $x,y \in B$ fixed, the map
\[
\phi :\R \ni s \mapsto \vert x+sy\vert_B\in \R
\]
is convex and therefore is  right and left differentiable. Let $D_{\pm}\vert x\vert  y $ denote the
right/left derivative of $\phi$ at $0$.
Let $B^\ast$ denote the dual of $B$ and let $\seq{\cdot,\cdot}$ denote the duality pairing between $B$ and $B^*.$
\begin{defi}
Let $x \in B.$ The \emph{sub-differential}
$\partial \vert x \vert $ of $\vert x \vert$ is defined as
\[
\partial \vert x \vert_B := \left\{ x^\ast \in B^\ast: D_{-}\vert x \vert y
\le \langle y, x^\ast \rangle \le D_{+}\vert x \vert y, \forall y \in
B\right\}.
\]
\end{defi}
%(for more details, see \cit{DaP-76}):
It can be proved, see e.g. \cite{dpz1}, that $\partial \vert x \vert$ is a nonempty, closed and convex set, and
\[
\partial \vert x \vert_B =\{ x^\ast \in B^\ast:
\langle x, x^\ast \rangle =\vert x \vert_B \; \;\mbox{\rm and } \vert x^\ast
\vert_{B^*} \le 1\}.
\]
In particular,
$\partial |0|_B$ is the unit ball in $B^\ast$. The following are the standing assumptions on the drift coefficient, the control set and the running and final cost functions,
\newpage
\begin{Assumption}\label{Assum3}
$\phantom{assumption}$
\begin{enumerate}
\item The control set $M$ is a metrisable \textbf{Suslin} space.

    \smallskip

\item The mapping $F:[0,T]\times B\times M\to B$ is continuous in every variable separately, uniformly with respect to $u\in M.$

    \smallskip

\item There exist
$k_1\in\R,k_2>0, \ m\geq 1$ and a measurable function
\[
\eta:[0,T]\times M\to [0,+\infty]
\]
such that for each $t\in [0,T],$ the mapping $\eta(t,\cdot):M\to[0,+\infty]$ is inf-compact and, for each $x\in D(A)$,  $y\in B$ and $u\in M$ we have
\begin{equation}\label{dar-2.1}
\seq{-A_Bx+F(t,x+y,u),z^*}\le -k_1|x|_B+k_2|y|^m_B+\eta(t,u), \ \ \text{for all} \ z^*\in\partial |x|_B.
\end{equation}

\smallskip
\item The \textbf{running cost} function $h:[0,T]\times B\times M\to (-\infty,+\infty]$ is measurable in $t\in [0,T]$ and lower semi-continuous with respect to $(x,u)\in B\times M.$

    \smallskip

\item There exist constants $C_1\in\R, \ C_2>0$ and
\begin{equation}\label{expcoercive}
\gamma>\frac{2m}{1-2(\delta+\sigma)}
\end{equation}
such that $h$ satisfies the \textbf{coercivity} condition,
\[
C_1+C_2\eta(t,u)^\gamma\le h(t,x,u), \ \ \ (t,x,u)\in [0,T]\times B\times M.
\]
\item The \textbf{final cost} function $\varphi:B\to\R$ is uniformly continuous.
\end{enumerate}
\end{Assumption}
We now state the main result of this paper on existence of weak optimal relaxed controls under the above assumptions.
\begin{theorem}
\label{Th:existence2}
Let $\bE$ be a separable UMD type-2 Banach space and let $A+\nu I$ satisfy Assumption \ref{Assum1} for some $\nu\geq 0.$ Suppose that $(A+\nu I)^{-1}$ is compact and that Assumptions \ref{Assum2} and \ref{Assum3} also hold. Let $x_0\in B$ be such that
\[
\inf_{{\pi}\in\bar{\U}_{\rm ad}^{\rm w}(x_0)}\bar J({\pi})<+\infty.
\]
Then \textbf{\emph{(RCP)}} admits a weak optimal relaxed control.
\end{theorem}
The proof of Theorem \ref{Th:existence2} is postponed until Section \ref{proofmain}. We illustrate the above result with the two following examples.
\begin{ex}[Optimal relaxed control of stochastic PDEs of reaction-diffusion type with multiplicative noise]
Let $M$ be a Suslin metrisable control set and let $\mathcal{O}\subset\R^d$ be a bounded domain with $\mathcal{C}^\infty$ boundary. Let
\[
f:[0,T]\times \mathcal{O}\times\R\times M\to \R
\]
be measurable in $t,$ continuous in $u$ and continuous in $(\xi,x)\in\mathcal{O}\times\R$ uniformly with respect to $u.$ Assume further that $f$ satisfies
\begin{equation}\label{pgc0}
f(t,\xi,x+y,u)\sgn x\le -k_1\abs{x}+k_2\abs{y}^m+\eta(t,u), \ \ \ (t,\xi)\in[0,T]\times\mathcal{O}, \ \ x,y\in\R, \ u\in M
\end{equation}
for some constants $m\geq 1$ and $k_1,k_2>0$ and some measurable function
\[
\eta:[0,T]\times M\to [0,+\infty]
\]
such that $\eta(t,\cdot)$ is inf-compact for all $t\in[0,T].$ Let $\mathcal{A}$ be the second-order differential operator
\[
(\mathcal{A}x)(\xi):=-\sum_{i,j=1}^d a_{ij}(\xi)\frac{\partial^2 x}{\partial \xi_i\partial \xi_j}+\sum_{i=1}^d b_{i}(\xi)\frac{\partial x}{\partial \xi_i} + c(\xi)x(\xi), \ \ \ \xi\in\mathcal{O},
\]
with $a_{ij}=a_{ji},$ $\sum_{i,j=1}^d a_{ij}(\xi)\lambda_i\lambda_j\geq C\abs{\lambda}^2, \forall\lambda\in\R^d,$ and $c,b_i,a_{ij}\in\mathcal{C}^\infty(\bar{\mathcal{O}}).$ Finally, let
\[
g:[0,T]\times \mathcal{O}\times \R \to\R
\]
be bounded, measurable in $t$ and continuous in $(\xi,x)\in\mathcal{O}\times \R,$ and consider the following controlled stochastic PDE on $[0,T]\times\mathcal{O},$
\begin{align}
\frac{\partial X}{\partial t}(t,\xi)+(\mathcal{A}X)(t,\xi)&=f(t,\xi,X(t,\xi),u(t))+g(t,\xi,X(t,\xi))\,\frac{\partial w}{\partial t}(t,\xi), \ \ \text{ on} \ [0,T]\times\mathcal{O}\notag\\
X(t,\xi)&=0, \ \ \ \ \ \ \ \ \ \ \ \ \ t\in (0,T], \ \xi\in\partial \mathcal{O}\label{spde1}\\
X(0,\cdot)&=x_0(\xi), \ \ \ \ \ \ \xi\in\mathcal{O}\notag
\end{align}
where $w(\cdot)$ is a nondegenerate noise with Cameron-Martin space
\[
\bH:=
\left\{
  \begin{array}{ll}
    L^2(\mathcal{O}), & \ \hbox{if} \ d=1, \\
    H^{\theta,2}(\mathcal{O}) \ \hbox{with } \ \theta\in\left(\frac{d-1}{2},\frac{d}{2}\right), & \ \hbox{if} \ d\geq 2.
  \end{array}
\right.
\]
In concrete situations, the quantity $X(t,\xi)$ may represent the concentration, density or temperature of a certain substance and, as mentioned in the introduction, we want to study a running cost function that permits to regulate this quantity at some fixed points $\zeta_1,\ldots,\zeta_n\in\mathcal{O}.$ Therefore, we need the trajectories of the state process to take values in the space of continuous functions on the domain $\mathcal{O}.$ In view of this, and the zero-boundary condition in (\ref{spde1}), we take $B=\mathcal{C}_0(\bar{\mathcal{O}})$ as state space.

Let $\phi:[0,T]\times\mathcal{O}\times \R\times M\to \R_+$ be measurable and lower semi-continuous with respect to $x$ and $u.$ We will consider the running cost function defined by
\begin{equation}\label{rch0}
h(t,x,u):=\sum_{i=1}^n \phi(t,\zeta_i,x(\zeta_i),u)+\eta(t,u)^\gamma, \ \ \ t\in[0,T], \ x\in\mathcal{C}_0(\bar{\mathcal{O}}), \ u\in M
\end{equation}
where $\zeta_1,\ldots,\zeta_n\in\mathcal{O}$ are fixed and $\gamma\geq 1$ is to be chosen below.
\end{ex}
\begin{theorem}
Let the constants $q,\sigma$ and $\delta$ satisfy the following conditions
\begin{enumerate}
\item If $d=1,$
\[\
q>2, \ \ \ \frac{1}{4}<\sigma<\frac{1}{2}-\frac{1}{2q} \ \ \ \hbox{and} \ \ \ \frac{1}{2q}<\delta<\frac{1}{2}-\sigma.
\]
\item If $d\geq 2,$
\[
2d<q\le\frac{2d}{d-2\theta}, \ \ \  \frac{d}{q}<\sigma<\frac{1}{4}, \ \ \ \hbox{and} \ \ \ \frac{d}{q}<\delta<\frac{1}{4}.
\]
\end{enumerate}
Assume also that $\gamma$ satisfies condition (\ref{expcoercive}). Then, if there exists $x_0\in \mathcal{C}_0(\bar{\mathcal{O}})$ such that
\[
\inf_{{\pi}\in\bar{\U}_{\rm ad}^{\rm w}(x_0)}\bar J({\pi})<+\infty,
\]
the \textbf{\emph{(RCP)}} associated with (\ref{spde1}) and the cost function (\ref{rch0}) admits a weak optimal relaxed control.
\end{theorem}
\begin{proof}
Let $\bE=L^q(\mathcal{O})$ and let $A_q$ denote the realization of $\mathcal{A}$ in $L^q(\mathcal{O}).$ Then $A_q+\nu I$ satisfies Assumption \ref{Assum1} for some $\nu\geq 0$ (see Example \ref{ex0}). Let us define the Nemytskii operator $F:[0,T]\times B\times M\to B$ by
\[
F(t,x,u)(\xi):=f(t,\xi,x(\xi),u), \ \ \ \xi\in\mathcal{O}, \ \ (t,x,u)\in [0,T]\times B\times M.
\]
Let $x\in B$ and let $z^*\in\partial\abs{x}_B.$ Then
\begin{equation}\label{zstar}
z^*=
\left\{
  \begin{array}{ll}
    \delta_{\xi_0}, & \ \hbox{if} \ x(\xi_0)=\abs{x}_B \\
    -\delta_{\xi_0}, & \ \hbox{if} \ x(\xi_0)=-\abs{x}_B
  \end{array}
\right.
\end{equation}
for some $\xi_0\in\mathcal{O}$ (see e.g. \cite{dpz5}) and by condition (\ref{pgc0}), for each $y\in B$ and $u\in M$ we have
\begin{align*}
\seq{F(t,x+y,u),z^*}&=f\left(t,\xi_0,x(\xi_0)+y(\xi_0),u)\right)\sgn x(\xi_0)\\
&\le -k_1\abs{x(\xi_0)}+k_2\abs{y(\xi_0)}^m+\eta(t,u)\\
%&\le K[1+\abs{y}_{B}^m+a(t,\abs{u}_B)]\\
&\le -k_1\abs{x}_{B}+k_2\abs{y}_{B}^m+\eta(t,u).
\end{align*}
Moreover, we can find $k_0\in\R$ such that the realization of $-\mathcal{A}+k_0I$ in $B$ is dissipative, i.e.
\[
\seq{(-\mathcal{A}+k_0I)x,x^*}\le 0, \ \ x^*\in\partial\abs{x}_B, \ \ \ x\in B.
\]
Hence, Assumption \ref{Assum3} is satisfied. We now check that Assumption \ref{Assum2}--(2) also holds. Observe that by writing $-A_q+F=-(A_q+\nu I)+F+\nu I,$ we can assume without loss of generality that $\nu=0.$ Let us define the multiplication operator
\[
\left(G(t,x)y\right)(\xi):=g(t,\xi,x(\xi))y(\xi), \ \ \ \xi\in \mathcal{O}, \ \ y\in \bH, \ \ x\in \mathcal{C}_0(\bar{\mathcal{O}}), \ \ t\in [0,T].
\]
We consider first the case $d=1.$ Since $g$ is bounded, for each $(t,x)\in [0,T]\times\mathcal{C}_0(\bar{\mathcal{O}})$ the map $\mathcal{O}\ni\xi\mapsto g(t,\xi,x(\xi))\in\R$ belongs to $L^\infty(\mathcal{O}).$ Therefore, by H\"{o}lder's inequality the map $G(t,x)$ is a bounded linear operator in $\bH=L^2(\mathcal{O})$ and its operator norm is uniformly bounded from above by some constant independent of $t$ and $x.$
Moreover, as recalled in Example \ref{ex2}, the map $A_q^{-\sigma}$ extends to a bounded linear operator from $L^2(\mathcal{O})$ to $L^q(\mathcal{O}),$ also denoted by $A_q^{-\sigma},$ such that
\[
A_q^{-\sigma}\in \gamma\left(L^2(\mathcal{O}),L^q(\mathcal{O})\right).
\]
Then, by the right-ideal property of the $\gamma-$radonifying operators, Assumption \ref{Assum2}--(2) is satisfied.

In the case $d\geq 2,$ the choice of the constants $\theta$ and $q$ and the Sobolev Embedding Theorem imply that $\bH=H^{\theta,2}(\mathcal{O})\hookrightarrow L^q(\mathcal{O}).$ This combined again with H\"{o}lder's inequality implies that $G(t,x)$ is a bounded linear operator from $H^{\theta,2}(\mathcal{O})$ into $L^q(\mathcal{O}),$ with operator norm again uniformly bounded from above by some constant independent of $t$ and $x.$ Since $\sigma>d/{2q},$ by the same argument used in Example \ref{ex1} it follows that
\[
A_q^{-\sigma}G(t,x)\in\gamma\bigl(H^{\theta,2}(\mathcal{O}),L^q(\mathcal{O})\bigr)
\]
and that Assumption \ref{Assum2}--(2) holds. Since in both cases $(d=1$ and $d\geq 2)$ we have $\delta>d/2q,$ as seen in Example \ref{ex1}, we have
 \[
D(A_q^\delta)=[L^q(\mathcal{O}),D(A_q)]_\delta=H^{2\delta,q}_0(\mathcal{O})\hookrightarrow\mathcal{C}_0(\bar{\mathcal{O}})
\]
and, therefore, Assumption \ref{Assum2}--(1) is satisfied too. Moreover, the last embedding is compact, which in turn implies that the embedding $D(A_q+\nu I)=D(A_q)\hookrightarrow L^q(\mathcal{O})$ is also compact, and the desired result follows from Theorem \ref{Th:existence2}.
\end{proof}
%For $\varepsilon>0$ small enough we have
%\[
%H_0^{2\delta,q}(\mathcal{O})\hookrightarrow W^{2\delta-\varepsilon,q}_0(\mathcal{O}).
%\]
%see e.g. \cite[p.180]{triebel}. If in addition $\varepsilon$ satisfies $(2\delta-\varepsilon)q>d$ then the embedding
%\[
%W_0^{2\delta-\varepsilon,q}(\mathcal{O})
%\]
%is compact which, in turn, implies that the embedding $D(A^\delta)\hookrightarrow\mathcal{C}_0(\bar{\mathcal{O}})$ is also compact.
\begin{rem}
Existence of weak optimal feedback controls for a similar cost functional and a similar class of dissipative stochastic PDEs has been recently proved in \cite{masiero5} and \cite{masiero4} using Backward SDEs and the associated Hamilton-Jacobi-Bellman equation. However, only the case of additive noise is considered and the nonlinear term is assumed to be bounded with respect to the control variable.
\end{rem}
\begin{ex}
The first example can be modified to allow the control process to be space-dependant. For instance,
consider the controlled stochastic PDE on $[0,T]\times\mathcal{O},$
\begin{align}
\frac{\partial X}{\partial t}(t,\xi)+(\mathcal{A}X)(t,\xi)&=f(t,\xi,X(t,\xi),u(t,\xi))+g(t,\xi,X(t,\xi))\,\frac{\partial w}{\partial t}(t,\xi), \ \ \text{ on} \ [0,T]\times\mathcal{O}\notag\\
X(t,\xi)&=0, \ \ \ \ \ \ \ \ \ \ \ \ \ t\in (0,T], \ \xi\in\partial \mathcal{O}\label{spde2}\\
X(0,\cdot)&=x_0(\xi), \ \ \ \ \ \ \xi\in\mathcal{O}\notag
\end{align}
where
\[
f:[0,T]\times \mathcal{O}\times\R\times\R\to \R
\]
satisfies
\begin{equation}\label{pgc1}
f(t,\xi,x+y,u)\sgn x\le -k_1\abs{x}+k_2\abs{y}^m+a(t,\abs{u}), \ \ \ (t,\xi)\in[0,T]\times\mathcal{O}, \ \ x,y,u\in\R
\end{equation}
where $a:[0,T]\times \R_+\to\R_+$ is a measurable function such that $a(t,\cdot)$ is strictly increasing for each $t\in [0,T].$ Assume further that $f$ is measurable in $t,$ separately continuous in $(\xi,u)\in\mathcal{O}\times \R$ and continuous in $x\in\R$ uniformly with respect to $\xi$ and $u.$

We take $M=\mathcal{C}(\bar{\mathcal{O}})$ as control set and fix $k,r>0$ such that $kr>d,$ in which case the embedding $H^{k,r}(\mathcal{O})\hookrightarrow\mathcal{C}(\bar{\mathcal{O}})$ is compact. Hence, as seen in Example \ref{infcompactex}, the mapping
\[
\eta:[0,T]\times M\to[0,+\infty]
\]
defined as
\[
\eta(t,u):=
\left\{
  \begin{array}{ll}
    a\bigl(t,c\abs{u}_{H^{k,r}(\mathcal{O})}\bigr), & \hbox{if} \ \ u\in H^{k,r}(\mathcal{O})\\
    +\infty, & \hbox{else}
  \end{array}
\right.
\]
satisfies
\[
\eta(t,\cdot) \ \ \text{ is inf-compact, for each } t\in [0,T].
\]
The constant $c>0$ in the definition of $\eta$ is such that $\abs{u}_{\mathcal{C}(\bar{\mathcal{O}})}\le c\abs{u}_{H^{k,r}(\mathcal{O})}, \ u \in H^{k,r}(\mathcal{O}).$
Finally, let
\[
\phi:[0,T]\times\mathcal{O}\times \R\times \R\to \R_+
\]
be measurable and lower semi-continuous with respect to $x,u\in\R,$ and define the running cost function
\begin{equation}\label{rch1}
h(t,x,u):=\sum_{i=1}^n \phi(t,\zeta_i,x(\zeta_i),u(\zeta_i))+\eta(t,u)^\gamma, \ \ \ t\in[0,T], \ x\in\mathcal{C}_0(\bar{\mathcal{O}}), \ u\in \mathcal{C}(\bar{\mathcal{O}})
\end{equation}
where $\zeta_1,\ldots,\zeta_n\in\mathcal{O}$ are fixed and $\gamma\geq 1$ is chosen to satisfy condition (\ref{expcoercive}).
The Nemytskii operator $F:[0,T]\times B\times M\to B$ is now defined as
\[
F(t,x,u)(\xi):=f(t,\xi,x(\xi),u(\xi)), \ \ \ \xi\in\mathcal{O}, \ \ (t,x,u)\in [0,T]\times B\times M.
\]
We see that Assumption \ref{Assum3} is again satisfied since for each $(t,x)\in [0,T]\times B$ and $z^*\in\partial\abs{x}_B$ as in (\ref{zstar}),
by condition (\ref{pgc1}), for all $y\in B$ and $u\in M$ we have
\begin{align*}
\seq{F(t,x+y,u),z^*}&=f\left(t,\xi_0,x(\xi_0)+y(\xi_0),u(\xi_0)\right)\sgn x(\xi_0)\\
&\le -k_1\abs{x(\xi_0)}+k_2\abs{y(\xi_0)}^m+a(t,\abs{u(\xi_0)})\\
&\le -k_1\abs{x}_B+k_2\abs{y}_B^m+\eta(t,u).
\end{align*}
\end{ex}

\section{Proof of the main Theorem}\label{proofmain}
We start by noticing that if $F$ satisfies the dissipative-type condition (\ref{dar-2.1}), since $\partial |0|_B$ coincides with the unit ball in $B^\ast,$ by the Hahn-Banach Theorem we have the following estimate
\begin{equation}\label{dar-2.3}
|F(t,y,u)|_B\le k_2|y|^m_B+\eta(t,u), \ \ \  \;t\in [0,T], \ \ y \in B, \ \ u\in M.
\end{equation}
For the proof of Theorem \ref{Th:existence2} we will also need the following important consequence of (\ref{dar-2.1}) in order to obtain a-priori estimates for weak admissible relaxed controls. Observe that the UMD property of the underlying Banach space and the bounded imaginary powers condition on $A$ turn out to be crucial for the proof of this Lemma.
\begin{lem}\label{Lem:a'priori}
Suppose that $F:[0,T]\times B\times M\to B$ satisfies Assumption \ref{Assum3}--(2) and that there exists a UMD Banach space $Y$ continuously embedded in $B$ such that the part $A_Y$ of the operator $A_B$ in $Y$ satisfies $A_Y\in\BIP^-(\frac{\pi}{2},Y).$ Suppose that a function $z\in\mathcal{C}([0,T];B)$ satisfies
\[
z(t)=\int^t_0 S_{t-r}\bar F(r,z(r)+v(r),q_r)\,dr, \ \ \ \ t\in [0,T].
\]
for some $\mathcal{P}(M)-$valued relaxed control $\set{q_t}_{t\geq 0}$ with
\[
\int_0^T\!\int_M\eta(t,u)^\gamma\,q_t(du)\,dt<\infty
\]
and some $v\in L^\gamma(0,T;B)$ with $\gamma>1.$ Then, $z$ satisfies the following estimate
\begin{equation}
|z(t)|_B\le \int^t_0e^{-k_1(t-s)}\Bigl[k_2|v(s)|^m_B+\int_M \eta(s,u)\,q_s(du)\Bigr]ds, \ \ \ \ t\in [0,T].
\label{eq:a'priori}
\end{equation}
\end{lem}
\begin{proof}
For $\lambda>0$ large enough define $R_\lambda:=\lambda(\lambda I+A_B)^{-1}\in \Lin(B),$ $z_\lambda(t):=R_\lambda z(t)$ and
\[
\bar{f}_\lambda(t):=R_\lambda \bar F(t,z(t)+v(t),q_t), \ \ \ \ t\in [0,T].
\]
Then $z_\lambda$ satisfies
\[
z_\lambda(t)=\int^t_0 S_{t-r}\bar f_\lambda(r)\,dr, \ \ \ \ t\in [0,T].
\]
Since $\norm{R_\lambda}_{\Lin(B,Y)}\le M$ for $\lambda>0$ large, we have $\bar f_\lambda \in L^\gamma(0,T;Y).$ Hence, by the Dore-Venni Theorem (see Theorem 3.2 in \cite{DV}), $z_\lambda\in W^{1,\gamma}(0,T;Y)\cap L^\gamma(0,T;D(A_Y))$ and satisfies in the $Y-$sense,
\[
\frac{dz_\lambda}{dt}(t)+A_Yz_\lambda(t)=\bar F(t,z_\lambda(t)+v(t),q_t)+\zeta_\lambda(t), \ \ \ \mbox{ for a.e. } \ t\in [0,T]\\
\]
with
\[
\zeta_\lambda(t):=\bar f_\lambda(t)-\bar F(t,z_\lambda(t)+v(t),q_t), \ \ \  \ t\in [0,T].
\]
Since $Y$ is continuously embedded in $B,$ the map $z_\lambda:[0,T]\to B$ is also a.e. differentiable and by Assumption \ref{Assum3}--(2) satisfies
\[
\frac{d^-}{dt}\abs{z_\lambda(t)}_B\le -k_1|z_\lambda(t)|_B+k_2|v(t)|^m_B+\zeta_{\lambda}(t)+\int_M\eta(t,u)\,q_t(du), \ \ \  \ t\in [0,T].
\]
Using Gronwall's Lemma it follows that
\[
|z_\lambda(t)|_B\le \int^t_0e^{-k_1(t-s)}\Bigl[k_2|v(s)|^m_B+\zeta_{\lambda}(s)+\int_M \eta(s,u)\,q_s(du)\Bigr]ds, \ \ \ \ t\in [0,T]
\]
and the result follows since $z_\lambda(t)\to z(t)$ for $t\in [0,T]$ and $\zeta_\lambda\to 0$ in $L^1(0,T;B).$
\end{proof}
\begin{rem}\label{aprioriY}
If $A\in\BIP^-(\frac{\pi}{2},\bE)$ then $D(A^\delta)\simeq[\bE,D(A)]_\delta$ is a UMD space (see e.g. \cite[Theorem 4.5.2]{amann}). Since the resolvent of $A$ commutes with $A^\delta,$ it follows that the realization of $A$ in $D(A^\delta)$ belongs to $\BIP^-(\frac{\pi}{2},D(A^\delta)).$ Hence, if Assumptions \ref{Assum2}-\ref{Assum3} are also satisfied then Lemma \ref{Lem:a'priori} applies with $Y=D(A^\delta).$
\end{rem}
%
%In what follows, we will occasionally use the  basic estimate
%\begin{equation}\label{ansem}
%\Vertt A^\delta S_t\Vertt_{\Lin(\bE)}\le m_\delta t^{-\delta}e^{-wt},
%\end{equation}
%valid for some constants $w\in\R$ and $m_\delta>0$ (see e.g. \cite[Theorem 2.6.13]{pazy}).
Finally, we will need the following measurability result.
\begin{lem}\label{gammatmeas}
Suppose that Assumption \ref{Assum3} is satisfied and the Banach space $B$ is separable. Define
\[
\Y^\gamma(0,T;M):=\left\{\lambda\in\Y(0,T;M):\int_0^T\!\int_M\eta(t,u)^\gamma\,\lambda(du,dt)<+\infty\right\}.
\]
Then, for each $t\in [0,T],$ the mapping $\Gamma_t:\mathcal{C}([0,T];B)\times\Y^\gamma(0,T;M)\to B$ defined by
\begin{equation}\label{gammat}
\Gamma_t(y,\lambda):=\int_0^t\int_M S_{t-r}F(r,y(r),u)\,\lambda(du,dr)
\end{equation}
is Borel-measurable.
\end{lem}
\begin{proof}
We fix $t\in [0,T]$ and prove first that the mapping $\Gamma_t(y,\cdot)$ is weakly-measurable for $y\in\mathcal{C}\left([0,T];B\right)$ fixed. Observe that for $x^*\in B^*$ fixed and $\lambda\in\Y^\gamma(0,T;M)$ we have
\begin{align*}
\seq{\Gamma_t(y,\lambda),x^*}&=\seq{\int_0^t\int_M S_{t-r}F(r,y(r),u)\,\lambda(du,dr),x^*}\\
&=\int_0^t\int_M \bigl\langle S_{t-r}F(r,y(r),u),x^* \bigr\rangle \,\lambda(du,dr).
\end{align*}
%where $\seq{\cdot,\cdot}$ denotes the duality pairing between $\bE$ and $\bE^*.$
%By Assumption \ref{Assum3}--(I) and
%\begin{align*}
%\abs{\seq{S_{t-r}F(r,y(r),u),x^*}}&=\abs{\seq{A^{\beta}S_{t-r}A^{-\beta}F(r,y(r),u),x^*}}\\
%&\le \norm{A^{\beta}S_{t-r}}_{\Lin(\bE)}\abs{A^{-\beta}F(r,y(r),u)}\abs{x^*}_{\bE^*}\\
%&\le T^{\beta}CK\left[1+\norm{y(\cdot)}_{\mathcal{C}([0,T];D(A^\delta))}+\eta(t,u)\right]
%\end{align*}
%for some constant $C=C(\beta,T,w).$
For each $N\in\mathds{N}$ define
\[
\phi_N(\lambda):=\int_0^t\int_M \min\{N,\seq{S_{t-r}F(r,y(r),u),x^*}\} \,\lambda(du,dr), \ \ \ \lambda\in\Y^\gamma(0,T;M).
\]
By Assumption \ref{Assum3}--(2), the integrand in the above expression is bounded and continuous with respect to $u\in M.$ Therefore, by Lemma \ref{Lem:contym} $\phi_N$ is continuous for each $N\in\mathds{N},$ and by the monotone convergence Theorem, $\phi_N(\lambda)\to \seq{\Gamma_t(y,\lambda),x^*}$ as $N\to\infty$ for all $\lambda\in\Y^\gamma(0,T;M).$ Hence, $\seq{\Gamma_t(y,\cdot),x^*}$ is measurable, i.e. $\Gamma_t(y,\cdot)$ is weakly-measurable. Since $B$ is separable, by the Pettis measurability Theorem (see \cite[Theorem 3.1.1]{showalter}), $\Gamma_t(y,\cdot)$ is also measurable.

Now, we prove that for $\lambda\in\Y^\gamma(0,T;M)$ fixed, the map $\Gamma_t(\cdot,\lambda)$ is continuous. Let $y_n\to y$ in $\mathcal{C}\left([0,T];B\right).$ Then, by Assumption \ref{Assum3}--(2) we have
\begin{align*}
&\bigr|S_{t-r}F(r,y(r),u)-S_{t-r}F(r,y_n(r),u)\bigr|_B\\
&\phantom{S_{t-r}}\le\norm{S_{t-r}}_{\Lin(B)}\bigr|F(r,y(r),u)-F(r,y_n(r),u)\bigr|_{B}\to 0
\end{align*}
as $n\to\infty$ for all $(r,u)\in [0,t]\times M.$ Moreover, for $\rho>0$ fixed there exists $\bar n\in\mathds{N}$ such that
\[
\sup_{r\in [0,T]}\bigl|y_n(r)-y(r)\bigr|_B < \rho \ \ \ \forall n\geq\bar N.
\]
Set $\rho':=\rho\vee \max_{1\le n\le\bar N-1}\sup_{r\in [0,T]}|y(r)-y_n(r)|_B.$ Then
\[
\sup_{r\in [0,T]}\bigl|y_n(r)-y(r)\bigr|_B < \rho', \ \ \ \forall n\in\mathds{N}
\]
and by (\ref{dar-2.3}), we have
\begin{align*}
\bigl|S_{t-r}&F(r,y(r),u)-S_{t-r}F(r,y_n(r),u)\bigr|_B\\
&\le \norm{S_{t-r}}_{\Lin(B)}\cdot \left[k_2\left(2^{m-1}\rho'^m+(2^{m-1}+1)\norm{y(\cdot)}_{\mathcal{C}([0,T];B)}^m\right)+\eta(r,u)\right].
\end{align*}
As $\eta$ belongs to $L^1([0,T]\times M;\lambda),$ so does the RHS of the above inequality. Therefore, by Lebesgue's dominated convergence Theorem we have
\begin{align*}
&\abs{\Gamma_t(y,\lambda)-\Gamma_t(y_n,\lambda)}_B\\
&\phantom{\Gamma_t}\le
\int_0^t\int_M\abs{S_{t-r}F(r,y(r),u)-S_{t-r}F(r,y_n(r),u)}_B\,\lambda(du,dr)\to 0
\end{align*}
as $n\to\infty,$ that is, $\Gamma_t(\cdot,\lambda)$ is continuous. Since $\Y^\gamma(0,T;M)$ is separable and metrisable, by Lemma 1.2.3 in \cite{cadfval} it follows that $\Gamma_t$ is jointly measurable.
\end{proof}
\begin{proof}[Proof of Theorem \ref{Th:existence2}]
Since we can write $-A+F=-(A+\nu I)+F+\nu I,$ by Remark \ref{nu0} we can assume without loss of generality that $\nu=0.$ Let
\[
\pi_n=\left(\Omega^n,\Fil^n,\mathds{P}^n,\mathds{F}_n,\{W_n(t)\}_{t\geq 0},\{q^n_t\}_{t\geq 0},\{X_n(t)\}_{t\geq 0}\right), \ \ n\in\mathds{N},\]
be a minimizing sequence of weak admissible relaxed controls, that is,
\[
\lim_{n\to\infty}\bar{J}(\pi_n)=\inf_{\pi\in\bar{\U}_{\rm ad}^{\rm w}(x)}\bar J({\pi}).
\]
From this and Assumption \ref{Assum3}-(5) it follows that there exists $R>0$ such that
\begin{equation}\label{estqn}
\Exp^n\int_0^T\!\!\int_M\eta(t,u)^{\gamma}\,q_t^n(du)\,dt\le-\frac{C_1}{C_2}+\frac{1}{C_2}\Exp^n\int_0^T\!\!\int_M h(t,X_n(t),u)\,q_t^n(du)\,dt\le R
\end{equation}
for all $n\in\mathds{N}.$

\noindent\textsc{Step 1.} Set $p=\frac{\gamma}{m}.$ Then, by (\ref{expcoercive}) $p>2$ and we can find $\alpha$ such that
%\[
%\frac{N}{q}<\frac{1}{p}<\frac12-\sigma-\delta
%\]
%and
\begin{equation}\label{delaphasigma}
\delta+\frac1p<\alpha<\frac12-\sigma.
\end{equation}
For each $n\in\mathds{N}$ define the process
\begin{align}
y_n(t)&:=\frac 1{\Gamma (1-\alpha )}\int^t_0(t-r)^{-\alpha}S_{t-r} G(r,X_n(r))\,dW_n(r), \ \ \ t\in [0,T].
\end{align}
Since, by Assumption \ref{Assum2} the mapping $A^{-\sigma}G:[0,\infty)\times B\times M \to\gamma(\bH,\bE)$ is bounded, from Lemma \ref{Lem:A3} we have
\begin{equation}\label{ynlp}
\sup_{n\ge 1}\,{\mathds E}^n\abs{y_n(\cdot)}^p_{L^p(0,T;\bE)}<+\infty .
\end{equation}
Then, by Chebyshev's inequality, the processes $\{y_n(\cdot)\}_{n\in\mathds{N}}$ are uniformly bounded in probability on $L^p(0,T;\bE).$ Since $A^{-1}$ is compact, it follows from Assumption \ref{Assum2}--(1), (\ref{delaphasigma}) and Corollary \ref{C:comp} that $\Lambda_T^{-\alpha}$ is a compact operator from $L^p(0,T;\bE)$ into $\mathcal{C}([0,T];B).$ Hence, the family of laws of the processes
\[
v_n:=\Lambda_T^{-\alpha}y_n, \ \  \ n\in\mathds{N},
\]
is tight on $\mathcal{C}([0,T];B)$. Now, for each $n\in\mathds{N}$ set $z_n:=\Lambda_T^{-1}f_n,$ i.e.
\begin{equation}\label{dar-2.61}
z_n(t)=\int_0^t S_{t-r}\bar F(r,X_n(r),q_r^n)\,dr, \ \ \ t\in [0,T].
\end{equation}
Then, by Theorem \ref{Th:2.1}, we have
\begin{equation}\label{Xn2}
X_n(t)=S_tx_0+z_n(t)+v_n(t), \ \ \Prob^n-\text{a.s.} \ \ t\in [0,T].
\end{equation}
Applying Lemma \ref{Lem:a'priori} to the process $z_n(\cdot)$ and (\ref{Xn2}) we obtain the estimate
\begin{equation}\label{yn}
|z_n(t)|_B\le \int^t_0e^{-k_1(t-s)}\Bigl[k_2|S_tx_0+v_n(s)|^m_B+\int_M \eta(s,u)\,q^n_s(du)\Bigr]ds, \ \  t\in [0,T].
\end{equation}
Moreover, by (\ref{ynlp}) and Lemma \ref{L:reg}, we have
\begin{equation}\label{estvn}
\sup_{n\in\mathds{N}}\, \Exp^n\biggl[\sup_{\, t\in [0,T]}|v_n(t)|_B^\zeta\biggr]
\le \sup_{n\in\mathds{N}}\, \Exp^n\biggl[\sup_{\, t\in [0,T]}|v_n(t)|_{D(A^\delta)}^\zeta\biggr]<+\infty, \ \ \forall \zeta\geq p.
\end{equation}
%By (\ref{estqn}) and H\"{o}lder's inequality, we also have
%\begin{equation}
%\sup_{n\in\mathds{N}}\Exp^n\int_0^T\!\!\int_M\eta(t,u)^{mp}\,q_t^n(du)\,dt<+\infty.
%\end{equation}
Using (\ref{estqn}), (\ref{yn}) and (\ref{estvn}) with $\zeta=m^2p$ we get
\[
\sup_{n\in\mathds{N}}\Exp^n\biggl[\sup_{\, t\in [0,T]}|z_n(t)|_B^{mp}\biggr]<+\infty.
\]
This, in conjunction with (\ref{Xn2}) and (\ref{estvn}) with $\zeta=mp,$ yields
\begin{equation}\label{ExpXn}
\sup_{n\in\mathds{N}}\Exp^n\biggl[\sup_{\, t\in [0,T]}|X_n(t)|_B^{mp}\biggr]<+\infty.
\end{equation}
Thus, by (\ref{dar-2.3}) and (\ref{ExpXn}), the processes %$\set{f_n(\cdot)}_{n\in\mathds{N}}$ defined by
\[
f_n(t):=\bar F(t,X_n(t),q_t^n), \ \ \ t\in [0,T], \ \ n\in\mathds{N},
\]
satisfy
\[
\sup_{n\in\mathds{N}}\Exp^n \abs{f_n(\cdot)}_{L^p(0,T;\bE)}^p<+\infty.
\]
This implies, again by Chebyshev's inequality, that the sequence of processes $\{f_n(\cdot)\}_{n\in\mathds{N}}$ is uniformly bounded in probability on $L^p(0,T;\bE).$ By compactness of the operator $\Lambda_T^{-1}$ and Corollary \ref{C:comp}, it follows that the family of the laws of $z_n=\Lambda_T^{-1}f_n,$ $n\in\mathds{N},$ is tight on $\mathcal{C}([0,T];B).$ By (\ref{Xn2}) we conclude that the family of laws of the processes $X_n,\ n\in\mathds{N},$ is also tight on $\mathcal{C}([0,T];B).$

Now, for each $n\in\mathds{N}$ we define a random Young measure $\lambda_n$ on $(\Omega^n,\Fil^n,\Prob^n)$ by the following formula
\begin{equation}\label{lambdan}
\lambda_n(du,dt):=q^n_t(du)\,dt.
\end{equation}
By (\ref{estqn}) and Lemma \ref{tightrym} the family of laws of $\{\lambda_n\}_{n\in\mathds{N}}$ is tight on $\Y(0,T;M).$ Hence, by Prohorov's Theorem, there exist a probability measure $\mu$ on $\mathcal{C}([0,T];B)^3\times\Y(0,T;M)$ and a subsequence of $\{X_n,z_n,v_n,\lambda_n\}_{n\in\mathds{N}},$ which we still denote by $\{X_n,z_n,v_n,\lambda_n\}_{n\in\mathds{N}},$ such that
\begin{equation}\label{lXn}
{\rm law}(X_n,z_n,v_n,\lambda_n)\to \mu, \ \ \ \text{weakly as} \ n\to\infty.
\end{equation}

\noindent\textsc{Step 2.} Since the space $\mathcal{C}([0,T];B)^3\times\Y(0,T;M)$ is separable and metrisable, using Dudley's generalization of the Skorohod representation Theorem (see e.g. Theorem 4.30 in \cite{kallen}) we ensure the existence of a probability space $(\tOmega,\tFil,\tProb)$ and a sequence of random variables $\{\tX_n,\tz_n,\tv_n,\tlambda_n\}_{n\in\mathds{N}}$ with values in $\mathcal{C}([0,T];B)^3\times\Y(0,T;M),$ defined on $(\tOmega,\tFil,\tProb),$ such that
\begin{equation}\label{eqlawsn}
(\tX_n,\tz_n,\tv_n,\tlambda_n)\stackrel{d}{=}(X_n,z_n,v_n,\lambda_n) , \ \ \ \text{for all }n\in\mathds{N},
\end{equation}
and, on the same stochastic basis $(\tOmega,\tFil,\tProb),$ a random variable $\bigl(\tX,\tz,\tv,\tlambda\bigr)$ with values in $\mathcal{C}([0,T];B)^3\times\Y(0,T;M)$ such that
\begin{equation}\label{tXn}
(\tX_n,\tz_n,\tv_n)\to(\tX,\tz,\tv), \ \ \ \text{in } \ \mathcal{C}([0,T];B)^3, \ \ \tProb-\text{a.s.}
\end{equation}
and
\begin{equation}\label{tlambdan}
\tlambda_n\to \tlambda, \ \ \ \text{stably in } \ \Y(0,T;M), \ \ \tProb-\text{a.s.}
\end{equation}
For each $t\in[0,T],$ let $\pi_t$ denote the evaluation map $\mathcal{C}([0,T];B)\ni z\mapsto z(t)\in B,$ and let $\Phi_t:\mathcal{C}([0,T];B)^2\times\Y^\gamma(0,T;M)\to B$ be the map defined by
\[
\Phi_t(x,z,\lambda):=\Gamma_t(x,\lambda)-\pi_t(z), \ \ \ (x,z)\in \mathcal{C}([0,T];B)^2, \ \lambda\in\Y^\gamma(0,T;M),
\]
with $\Gamma_t$ as in (\ref{gammat}). By Lemma \ref{gammatmeas}, the map $\Phi_t$ is measurable. Hence, by (\ref{dar-2.61}) and (\ref{eqlawsn}), for each $t\in [0,T]$ we have
\begin{equation}\label{tzn0}
\tz_n(t)=\int_0^t \int_M S_{t-r} F(r,\tX_n(r),u)\,\tlambda_n(du,dr), \ \ \ \tProb-\text{a.s.}
\end{equation}
A similar argument used with (\ref{Xn2}) and (\ref{eqlawsn}) yields
\begin{equation}\label{tXn2}
\tX_n(t)=S_t x_0+\tz_n(t)+\tv_n(t), \ \ \ \tProb-\text{a.s.}, \ \ \ \ \ t\in [0,T].
\end{equation}
Moreover, by Lemma \ref{Lem:disrym}, for each $n\in\mathds{N}$ there exists a relaxed control process $\{\tq^n_t\}_{t\geq 0}$ defined %\in\mathcal{R}(0,T;M)$
on $(\tOmega,\tFil,\tProb)$ such that
\begin{equation}\label{tqn}
\tlambda_n(du,dt)=\tq^n_t(du)\,dt, \ \ \ \tProb-\mbox{a.s.}
\end{equation}
Thus, we can rewrite (\ref{tzn0}) as
\begin{equation}\label{tzn1}
\tz_n(t)=\int_0^t S_{t-r} \bar F(r,\tX_n(r),\tq_r^n)\,dr, \ \ \ \tProb-\text{a.s.}, \ \ \ t\in [0,T].
\end{equation}
\textsc{Step 3.} For each $n\in\mathds{N},$ we define an $\bE-$valued process $\tM_n(\cdot)$ by
\begin{equation}\label{tMn1}
\tM_n(t):=A^{-1}\tX_n(t)+\int_0^t \tX_n(r)\,dr-A^{-1}x_0-\int_0^t A^{-1}\bar F(r,\tX_n(r),\tq_r^n)\,dr, \ \ \ \ t\in [0,T],
\end{equation}
and a filtration $\tilde{\mathds{F}}_n:=\bigl\{\tilde{\Fil}_t^n\bigr\}_{t\in [0,T]}$ by
\[
\tilde{\Fil}_t^n:=\sigma\{(\tX_n(s),\tq^n_s):0\le s\le t\}, \ \ \ t\in[0,T].
\]
%We have now the following Lemma.
\begin{lem}\label{tMnmart}
The process $\tM_n(\cdot)$ is a $\mathds{\tilde{F}}_n-$martingale with cylindrical quadratic variation
\[
[\tilde{M}_n](t)=\int_0^t\tilde{Q}_n(s)\,ds, \ \ \ \ t\in [0,T].
\]
where $\tilde{Q}_n(t):=[A^{-1} G(t,\tX_n(t))]\circ[A^{-1}G(t,\tX_n(t))]^*\in\Lin(\bE^*,\bE).$ Moreover, $\tM_n(\cdot)$ satisfies
\begin{equation}\label{tMn2}
\tM_n(t)=A^{-1}\tv_n(t)+\int_0^t \tv_n(s)\,ds, \ \ \ \tProb-\text{a.s.}, \ \ \ t\in [0,T].
\end{equation}
\end{lem}
\begin{proof}[Proof of Lemma \ref{tMnmart}] %since the semigroup $(S_t)_{t\geq 0}$ generated by $-A$ is analytic
By Theorem 2.6.13 in \cite{pazy}, for each $t\geq 0$ we have $A^{-1}S_t=S_tA^{-1}.$ Therefore, the process $X_n(\cdot)$ satisfies
\[
A^{-1}X_n(t)=S_t A^{-1}x_0+\int_0^t S_{t-r}A^{-1}\bar F(r,X_n(r),q^n_r)\,dr+\int_0^t S_{t-r} A^{-1}G(r,X_n(r))\,dW_n(r).
\]
Since $1-\sigma>\frac12,$ we have (see e.g. \cite{triebel}),
\[
{\rm Range} (A^{-(1-\sigma)})=D(A^{1-\sigma})\subset D_A(\textstyle{\frac12},2).
\]
 Therefore, by the left-ideal property of the $\gamma-$radonifying operators, for each $(t,x)\in [0,T]\times D(A^\delta)$ we get
\[
A^{-1}G(t,x)=A^{-(1-\sigma)}A^{-\sigma}G(t,x)\in \gamma\left(\bH,D_A(\textstyle{\frac12},2)\right).
\]
Similarly, we see that $A^{-1}x\in D(A^{1-\delta})\subset D_A(\textstyle{\frac12},2).$ Thus, from Lemma \ref{Lem:strong}, it follows
\[
A^{-1}X_n(t)+\int_0^t X_n(s)\,ds=A^{-1}x_0+\int_0^t A^{-1}\bar F(r,X_n(r),q^n_r)\,dr+\int_0^t A^{-1}G(r,X_n(r))\,dW_n(r).
\]
%$\Prob^n-\text{a.s.}, \ t\in[0,T].$
Then, for each $n\in\mathds{N},$ the $\bE-$valued he process $M_n(\cdot)$ defined as
\[
M_n(t):=A^{-1}X_n(t)+\int_0^t X_n(s)\,ds-A^{-1}x_0-\int_0^t A^{-1}\bar F(r,X_n(r),q^n_r)\,dr, \ \ \ \ t\geq 0,
\]
is a $\mathds{F}_n-$martingale with cylindrical quadratic variation
\[
[M_n](t)=\int_0^t Q_n(s)\,ds, \ \ \ \ t\geq 0,
\]
where $Q_n(t)=[A^{-1} G(t,X_n(t))]\circ[A^{-1} G(t,X_n(t))]^*\in\Lin(\bE^*,\bE).$ Clearly, the process $M_n(\cdot)$ is adapted to the filtration generated by the process $(X_n,q^n),$ which in turn is adapted to $\mathds{F}_n.$ Then, the first part of the Lemma follows since
\[
(X_n,q^n)\stackrel{d}{=}(\tX_n,\tq^n).
\]
Now, by (\ref{tzn1}) and Lemma \ref{Lem:mildclassical} in the Appendix, the process $\tz_n(\cdot)$ satisfies
%Theorem 2.4 and Corollary 2.6 in \cite{pazy}, %by Lemma \ref{Lem:mildclassical} in the Appendix,
\begin{equation}\label{tzn2}
A^{-1}\tz_n(t)+\int_0^t \tz_n(s)\,ds=\int_0^t A^{-1}\tf_n(s)\,ds, \ \ \ \tProb-\text{a.s.}, \ \ \ t\in [0,T],
\end{equation}
with
\begin{equation}\label{tfn}
\tf_n(t):=\bar F(t,\tX_n(t),\tq_t^n), \ \ \  t\in [0,T].
\end{equation}
Thus, using (\ref{tXn2}) in (\ref{tzn2}) we get
\begin{equation}\label{tXn3}
A^{-1}\tX_n(t)-A^{-1}S_tx_0-A^{-1}\tv_n(t)+\int_0^t(\tX_n(r)-S_r x_0-\tv_n(r))\,dr=\int_0^t A^{-1}\tf_n(s)\,ds.
\end{equation}
From the identity (see e.g. \cite[Theorem 1.2.4]{pazy}),
\[
-A\int_0^t S_rx_0\,dr=S_tx_0-x_0
\]
we get
\[
\int_0^t S_rx_0\,dr+A^{-1}S_t x_0=A^{-1}x_0.
\]
Using this in (\ref{tXn3}) and rearranging the terms we obtain
\[
A^{-1}\tX_n(t)+\int_0^t\tX_n(r)\,dr=A^{-1}x_0+\int_0^t A^{-1}\tf_n(s)\,ds+A^{-1}\tv_n(t)+\int_0^t\tv_n(r)\,dr, \ \ \ \tProb-\text{a.s.},
\]
for $t\in [0,T]$ and so, in view of (\ref{tMn1}) and (\ref{tfn}), equality (\ref{tMn2}) follows.
\end{proof}
\noindent\textsc{Step 4.} We now define a $\bE-$valued process $\tM(\cdot)$ by the formula
\[
\tM(t):=A^{-1}\tv(t)+\int_0^t \tv(s)\,ds, \ \ \ \ t\in [0,T].
\]
Observe that from (\ref{tXn}) we have $\tv_n\to\tv$ in $\mathcal{C}([0,T];\bE)$ $\tProb-$a.s. which combined with (\ref{tMn2}) yields
\begin{equation}\label{tMnconv}
\tM_n\to \tM, \ \ \ \text{ in } \ \mathcal{C}([0,T];\bE), \ \ \ \tProb-\text{a.s.}
\end{equation}
We use once again Lemma \ref{Lem:disrym} to ensure existence of a relaxed control process $\{\tq_t\}_{t\geq 0}$ %\in\mathcal{R}(0,T;M)$
defined on $(\tOmega,\tFil,\tProb)$ such that
\begin{equation}\label{tqLG}
\tlambda(du,dt)=\tq_t(du)\,dt, \ \ \ \tProb-\mbox{a.s.}
\end{equation}
We define the filtration $\tilde{\mathds{F}}:=\bigl\{\tFil_t\bigr\}_{t\in [0,T]}$ by
\[
\tFil_t:=\sigma\{(\tX(s),\tq_s):0\le s\le t\}, \ \ \ \ t\in [0,T].
\]
Let also $\tilde{g}(t):= G(t,\tX(t))$ and $\tilde{Q}(t):=[A^{-1}\tilde{g}(t)]\circ[A^{-1}\tilde{g}(t)]^*\in\Lin(\bE^*,\bE),$ $t\in [0,T].$
\begin{lem}
\label{tMmartLG}
The process $\tM(\cdot)$ is a $\tilde{\mathds{F}}-$ martingale with cylindrical quadratic variation
\[
[\tilde{M}](t)=\int_0^t\tilde{Q}(s)\,ds, \ \ \ \ t\in [0,T].
\]
\end{lem}
\begin{proof}[Proof of Lemma \ref{tMmartLG}]
From (\ref{tXn}) we have
\begin{equation}\label{tXn4}
\sup_{\, t\in[0,T]}\bigl|\tX_n(t)-\tX(t)\bigr|_B^2\to 0, \ \ \text{ as } n\to \infty, \ \ \tProb-\text{a.s.}
\end{equation}
Moreover, by (\ref{ExpXn}), (\ref{eqlawsn}) and Fatou's Lemma it follows that
\begin{equation}\label{tXest}
\tExp\biggl[\sup_{\, t\in [0,T]}\bigl|\tX(t)\bigr|_B^{mp}\biggr]<+\infty.
\end{equation}
Therefore, by (\ref{ExpXn}) and Chebyshev's inequality, the random variables in (\ref{tXn4}) are uniformly integrable, and by \cite[Lemma 4.11]{kallen} we have
\begin{equation}\label{tXnconv}
\tExp\biggl[\sup_{\, t\in[0,T]}\bigl|\tX_n(t)-\tX(t)\bigr|_B^2\biggr]\to 0, \ \ \text{ as } n\to \infty.
\end{equation}
The same argument applied to (\ref{tMnconv}) yields %and \ref{dar-2.3}
\begin{equation}\label{tMnconv2}
\tExp\biggl[\sup_{\, t\in[0,T]}\bigl|\tM_n(t)-\tM(t)\bigr|_\bE^2\biggr]\to 0, \ \ \text{ as } n\to \infty.
\end{equation}
This, in conjunction with Lemma \ref{tMnmart}, implies that for all $0<s<t$ and for all
\[
\phi\in \mathcal{C}_b\left(\mathcal{C}(0,s;B)\times\Y(0,s;M)\right)
\]
we have
\[
0=\tExp\left[\bigl(\tM_n(t)-\tM_n(s)\bigr)\phi(\tX_n,\tlambda_n)\right]\to \tExp\left[\bigl(\tM(t)-\tM(s)\bigr)\phi(\tX,\tlambda)\right]
\]
as $n\to\infty,$ which implies that $\tM(\cdot)$ is a $\tilde{\mathds{F}}-$martingale. Moreover, for all $x_1^*,x_2^*\in\bE$ and $n\in\mathds{N},$
\begin{equation}\label{tcqv}
\begin{split}
\tExp\biggl[&\Bigl(\bigl\langle\tM_n(t),x_1^*\bigr\rangle\bigl\langle\tM_n(t),x_2^*\bigr\rangle
-\bigl\langle\tM_n(s),x_1^*\bigr\rangle\bigl\langle\tM_n(s),x_2^*\bigr\rangle\biggr.\\
&\biggl.
-\int_s^t\bigl[\bigl(A^{-1}G(r,\tX_n(r))\bigr)^*x_1^*,\bigl(A^{-1}G(r,\tX_n(r))\bigr)^*x_1^*\bigr]_\bH\,dr\Bigr)\phi(\tX_n,\tlambda_n)\biggr]=0.
\end{split}
\end{equation}
By (\ref{estvn}) and  (\ref{tMn2}), the first two terms inside the expectation in (\ref{tcqv}) are uniformly integrable, and so is the third term by Assumption \ref{Assum2}. Hence, by (\ref{tXnconv}), (\ref{tMnconv2}) and the continuity of $A^{-1}G,$ the limit of (\ref{tcqv}) as $n\to\infty$ yields
\begin{align*}
\tExp\biggl[&\Bigl(\bigl\langle\tM(t),x_1^*\bigr\rangle\bigl\langle\tM(t),x_2^*\bigr\rangle
-\bigl\langle\tM(s),x_1^*\bigr\rangle\bigl\langle\tM(s),x_2^*\bigr\rangle\biggr.\\
&\biggl.-\int_s^t\bigl[\bigl(A^{-1}G(r,\tX(r))\bigr)^*x_1^*,\bigl(A^{-1}G(r,\tX(r))\bigr)^*x_1^*\bigr]_\bH\,dr\Bigr)\phi(\tX,\tlambda)\biggr]=0,
\end{align*}
and Lemma \ref{tMmartLG} follows.
\end{proof}
\noindent{\sc Step 5.} The aim of this and the next step is to identify $\{\tX(t)\}_{t\geq 0}$ as mild solution of the equation controlled by $\set{\tq_t}_{t\geq 0}.$ Notice that the coercivity condition in Assumption \ref{Assum3}--(5), which we used before to obtain the uniform estimates for the minimizing sequence, will again be essential to pass to the limit as the nonlinearity is not necessarily bounded with respect to the control variable.

Let $\tf(t):=\bar F(t,\tX(t),\tq_t),$ $t\in [0,T].$ Observe that, by (\ref{dar-2.3}), (\ref{tfn}) and (\ref{tXest}), $\tf_n,\tf$ belong to $L^2([0,T]\times\tOmega;B).$ We claim first that
\begin{equation}\label{tfntf}
\tf_n\to \tf, \ \ \ \text{ weakly in } \ L^2([0,T]\times\tOmega;\bE).
\end{equation}
\begin{proof}[Proof of (\ref{tfntf})]
For each $n\in\mathds{N},$ define $\hat f_n(t):=\bar F(t,\tX(t),\tq_t^n),$ $t\in [0,T].$ First, we will prove
\begin{equation}\label{tfnhatfn}
\tf_n-\hat f_n\to 0, \ \ \ \ \text{(strongly) in} \ L^2([0,T]\times\tOmega;B).
\end{equation}
Indeed, by Assumption \ref{Assum3}, we have
\begin{align*}
I_n(t):&=\int_M\bigl|F(t,\tX_n(t),u)-F(t,\tX(t),u)\bigr|_B^2\,\tq^n_t(du)\\
&\le\sup_{u\in M}\bigr|F(t,\tX_n(t),u)-F(t,\tX(t),u)\bigr|_B^2\to 0
\end{align*}
as $n\to\infty$ for $t\in[0,T],$ $\tProb-$a.s. From (\ref{dar-2.3}), (\ref{estqn}), (\ref{ExpXn}) and (\ref{tXest}) we have
\[
\sup_{n\in\mathds{N}}\tExp\int_0^T\abs{I_n(t)}^{p/2}\,dt<+\infty.
\]
Hence, $\{I_n(\cdot)\}_{n\in\mathds{N}}$ is uniformly integrable on $\tOmega\times [0,T],$ and by \cite[Lemma 4.11]{kallen} we have
\[
\tExp\int_0^T\bigl|\tf_n(t)-\hat f_n(t)\bigr|^2_B\,dt\le\tExp\int_0^T\!I_n(t)\,dt\to 0, \ \ \ \text{ as } \ n\to\infty,
\]
and (\ref{tfnhatfn}) follows.

Now, we will prove that
\begin{equation}\label{hatfntf}
\hat f_n\to \tf, \ \ \ \text{ weakly in } \ L^2([0,T]\times\tOmega;\bE).
\end{equation}
Since $\bE$ is separable and reflexive, the dual space $\bE^*$ is also separable and, therefore, has the Radon-Nikodym property with respect to the product measure $dt\otimes d\Prob$ (see e.g. Sections III.2 and IV.2 in \cite{diesteluhl}) and so we have
\[
L^2([0,T]\times\tOmega;\bE)^*\simeq L^2([0,T]\times\tOmega;\bE^*).
\]
Let $\psi\in L^2([0,T]\times\tOmega;\bE^*)$ be fixed, and observe that for each $n\in\mathds{N},$
\begin{align*}
\tExp\int_0^T\bigl\langle \hat f_n(t),\psi(t)\bigr\rangle\,dt&=\tExp\int_0^T\seq{\int_M F(t,\tX(t),u)\,\tq^n_t(du),\psi(t)}\,dt\\
&=\tExp\int_0^T\int_M \bigl\langle F(t,\tX(t),u),\psi(t)\bigr\rangle\,\tq^n_t(du)dt\\
&=\tExp\int_0^T\int_M \bigl\langle F(t,\tX(t),u),\psi(t)\bigr\rangle\,\tlambda^n(du,dt).
\end{align*}
where $\seq{\cdot,\cdot}$ denotes the duality pairing between $\bE$ and $\bE^*.$ Let $\varepsilon\in(0,1)$ be fixed and take $C_\varepsilon>\max\{\frac{R}{\varepsilon},1\}$ with $R$ as in (\ref{estqn}). Then, for this choice of $C_\varepsilon,$ we have
\begin{align*}
\tExp\left[\tlambda_n\Bigl(\bigl\{\eta^{\gamma-2}>C_\varepsilon\bigr\}\Bigr)\right]
&=\tExp\int_{\set{\eta(t,u)^{\gamma-2}>C_\varepsilon}}\tlambda_n(du,dt)\\
&\le\frac{1}{C_\varepsilon}\tExp\int_{\{\eta(t,u)^{\gamma-2}>C_\varepsilon\}}\eta(t,u)^{\gamma-2}\,\tlambda_n(du,dt)\\
&<\varepsilon.
\end{align*}
We write
\begin{align*}
\tExp\int_0^T&\!\int_M\bigl\langle F(t,\tX(t),u),\psi(t)\bigr\rangle\,\tlambda_n(du,dt)\\
=\,&\tExp\int_{\{\eta(t,u)^{\gamma-2}\le C_\varepsilon\}}\bigl\langle F(t,\tX(t),u),\psi(t)\bigr\rangle\,\tlambda_n(du,dt)\\
&+\tExp\int_{\{\eta(t,u)^{\gamma-2}> C_\varepsilon\}}\bigl\langle F(t,\tX(t),u),\psi(t)\bigr\rangle\,\tlambda_n(du,dt)
\end{align*}
and observe first that by Lemma \ref{Lem:contym} we have $\tProb-$a.s.
\begin{align*}
\int_{\{\eta(t,u)^{\gamma-2}\le C_\varepsilon\}}&\bigl\langle F(t,\tX(t),u),\psi(t)\bigr\rangle\,\tlambda_n(du,dt)\\
&\to \int_{\{\eta(t,u)^{\gamma-2}\le C_\varepsilon\}}\bigl\langle F(t,\tX(t),u),\psi(t)\bigr\rangle\,\tlambda(du,dt)
\end{align*}
as $n\to\infty$ and that, by (\ref{dar-2.3}),
\begin{align*}
\int_{\{\eta(t,u)^{\gamma-2}\le C_\varepsilon\}}&\bigl\langle F(t,\tX(t),u),\psi(t)\bigr\rangle\,\tlambda_n(du,dt)\\
&\le k_2T\left(\bigl|\!\bigl|\tX(\cdot)\bigr|\!\bigr|_{\mathcal{C}([0,T];B)}^m+C_\varepsilon\right)\abs{\psi(\cdot)}_{L^2(0,T;\bE^*)}, \ \ \tProb-\text{a.s.}
\end{align*}
Thus, using Lebesgue's dominated convergence Theorem we get
\begin{align*}
\tExp\int_{\{\eta(t,u)^{\gamma-2}\le C_\varepsilon\}}&\bigl\langle  F(t,\tX(t),u),\psi(t)\bigr\rangle\,\tlambda_n(du,dt)\\
&\to \tExp\int_{\{\eta(t,u)^{\gamma-2}\le C_\varepsilon\}}\bigl\langle F(t,\tX(t),u),\psi(t)\bigr\rangle\,\tlambda(du,dt)
\end{align*}
as $n\to\infty.$ Now, for each $n\in\mathds{N},$ define the measure $\mu_n$ on $\B(M)\otimes\B([0,T])\otimes\Fil$ as
\[
\mu_n(du,dt,d\omega):=\tlambda_n(\omega)(du,dt)\tProb(d\omega).
\]
Then, again by (\ref{dar-2.3}), for each $n\in\mathds{N}$ we have
\begin{align*}
\tExp \int_{\{\eta(t,u)^{\gamma-2}> C_\varepsilon\}}&\Bigl|\bigl\langle F(t,\tX(t),u),\psi(t)\bigr\rangle\Bigr|\,\tlambda_n(du,dt)\\
\le&\int_{\tOmega}\int_{\{\eta(t,u)^{\gamma-2}> C_\varepsilon\}}\varphi(t)\,\mu_n(du,dt,d\omega)\\
&+\int_{\tOmega}\int_{\{\eta(t,u)^{\gamma-2}> C_\varepsilon\}}\eta(t,u)\abs{\psi(t)}_{\bE^*}\,\mu_n(du,dt,d\omega)
\end{align*}
with $\varphi:=k_2\abs{X(\cdot)}^m_\bE\abs{\psi(\cdot)}_{\bE^*}\in L^r([0,T]\times\tOmega)$ and $\frac{1}{2}+\frac{1}{p}=\frac{1}{r},$ since by (\ref{tXest}) we have $\abs{X(\cdot)}_B^{m}\in L^p([0,T]\times\tOmega).$ Thus, by H\"{o}lder's inequality we get
\begin{align*}
\int_{\tOmega}&\int_{\{\eta(t,u)^{\gamma-2}> C_\varepsilon\}}\varphi(t)\,\mu_n(du,dt,d\omega)\\
&\le \left(\int_{\tOmega}\int_0^T\!\int_M\varphi(t)^r\,\mu_n(du,dt,d\omega)\right)^{1/r}
\cdot\left(\tExp\left[\tlambda_n\left(\eta^{\gamma-2}>C_\varepsilon\right)\right]\right)^{1-1/r}\\
&<\norm{\varphi}_{L^r([0,T]\times\tOmega)}\varepsilon^{1-1/r}
\end{align*}
and
\begin{align*}
\int_{\tOmega}&\int_{\{\eta(t,u)^{\gamma-2}>C_\varepsilon\}}\eta(t,u)\abs{\psi(t)}_{\bE^*}\,\mu_n(du,dt,d\omega)\\
&\le\norm{\psi}_{L^2([0,T]\times\tOmega;\bE^*)}
\left(\int_{\tOmega}\int_{\{\eta(t,u)^{\gamma-2}>C_\varepsilon\}}\eta(t,u)^\gamma\,\mu_n(du,dt,d\omega)\right)^{1/2}\\
&=\norm{\psi}_{L^2([0,T]\times\tOmega;\bE^*)}
\left(\tExp\int_{\{\eta(t,u)^{\gamma-2}>C_\varepsilon\}}\frac{\eta(t,u)^\gamma}{\eta(t,u)^{\gamma-2}}\,\tlambda_n(du,dt)\right)^{1/2}\\
&\le\norm{\psi}_{L^2([0,T]\times\tOmega;\bE^*)}
\left(\frac{1}{C_\varepsilon}\tExp\int_{\{\eta(t,u)^{\gamma-2}>C_\varepsilon\}}\eta(t,u)^\gamma\,\tlambda_n(du,dt)\right)^{1/2}\\
&\le \norm{\psi}_{L^2([0,T]\times\tOmega;\bE^*)}\left(\frac{R}{C_\varepsilon}\right)^{1/2}\\
&<\norm{\psi}_{L^2([0,T]\times\tOmega;\bE^*)}\varepsilon^{1/2},
\end{align*}
and this holds uniformly with respect to $n\in\mathds{N}.$ Since $\eta(t,\cdot)$ is lower semi-continuous for all $t\in [0,T],$ by Lemma \ref{Lem:lscym} and Fatou's lemma we have
\[
\tExp\int_0^T\int_M\eta(t,u)^\gamma\,\tlambda(du,dt)\le \liminf_{n\to \infty} \tExp\int_0^T\int_M\eta(t,u)^\gamma\,\tlambda_n(du,dt)\le R.
\]
Therefore, the same estimate holds for $\tlambda,$ that is,
\begin{align*}
\tExp\int_{\{\eta(t,u)^{\gamma-2}> C_\varepsilon\}}&\Bigl|\bigl\langle  F(t,\tX(t),u),\psi(t)\bigr\rangle\Bigr|\,\tlambda(du,dt)\\
&\le \norm{\varphi}_{L^r([0,T]\times\tOmega)}\varepsilon^{1-1/r}+\norm{\psi}_{L^2([0,T]\times\tOmega;\bE^*)}\varepsilon^{1/2}
\end{align*}
and since $\varepsilon\in(0,1)$ is arbitrary, we conclude that
\[
\tExp\int_0^T\!\int_M\bigl\langle F(t,\tX(t),u),\psi(t)\bigr\rangle\,\tlambda_n(du,dt)\to \tExp\int_0^T\!\int_M\bigl\langle F(t,\tX(t),u),\psi(t)\bigr\rangle\,\tlambda(du,dt)
\]
as $n\to\infty.$ Thus, (\ref{hatfntf}) follows. Note that (\ref{hatfntf}) in conjunction with (\ref{tfnhatfn}) implies (\ref{tfntf}).
\end{proof}
\noindent\textsc{Step 6.} We now claim that the process $\tM(\cdot)$ satisfies, for each $t\in [0,T],$
\begin{equation}\label{tM2LG}
\tM(t)=A^{-1}\tX(t)+\int_0^t \tX(s)\,ds-A^{-1}x_0-\int_0^tA^{-1} \bar F(s,\tX(s),\tq_s)\,ds, \ \ \tProb-\text{a.s.}
\end{equation}
\begin{proof}[Proof of (\ref{tM2LG})]
By (\ref{tXnconv}) and (\ref{tMnconv2}), for any $\varepsilon>0$ there exists an integer $\bar{m}=\bar{m}(\varepsilon)\geq 1$ for which
\begin{equation}\label{convXM1}
\tExp\biggl[\sup_{\, t\in [0,T]}\bigl|\tX_n(t)-\tX(t)\bigr|_{B}^2+\bigl|\tM_n(t)-\tM(t)\bigr|^2_{\bE}\biggr]<\varepsilon, \ \ \ \forall n\geq\bar m.
\end{equation}
From (\ref{tfntf}) we have
\[
\tf\in\overline{\{\tf_{\bar{m}},\tf_{\bar{m}+1},\ldots\}}^{w}
\subset \overline{{\rm co}\{\tf_{\bar{m}},\tf_{\bar{m}+1},\ldots\}}^{w}
\]
where ${\rm co}(\cdot)$ and $\overline{\,\cdot\,}^{w}$ denote the convex hull and weak-closure in $L^2([0,T]\times\tOmega;\bE)$ respectively. By Mazur's theorem (see e.g. \cite[Theorem 2.5.16]{meg}),
\[
\overline{{\rm co}\{\tf_{\bar{m}},\tf_{\bar{m}+1},\ldots\}}^{w}
=\overline{{\rm co}\{\tf_{\bar{m}},\tf_{\bar{m}+1},\ldots\}}.
\]
Therefore, there exist an integer $\bar{N}\geq 1$ and $\{\alpha_1,\ldots,\alpha_{\bar{N}}\}$ with $\alpha_i\geq 0,$ $\sum_{i=1}^{\bar{N}}\alpha_i=1,$ such that
\begin{equation}\label{alphaif1}
\Bigl|\!\Bigl|\sum_{i=1}^{\bar{N}}\alpha_i\tf_{\bar m+i}-\tf\Bigr|\!\Bigr|^2_{L^2([0,T]\times\tOmega;\bE)}<\varepsilon.
\end{equation}
Let $t\in [0,T]$ be fixed. Using the $\alpha_i's$ and the definition of the process $\tM_{\bar m+i}$ in (\ref{tMn1}) we can write
\[
A^{-1}x_0=\sum_{i=1}^{\bar{N}}\alpha_i\left[A^{-1}\tX_{\bar m+i}(t)+\int_0^t\tX_{\bar m+i}(s)\,ds-\int_0^t A^{-1}\tf_{\bar m+i}(s)\,ds-\tM_{m+i}(t)\right]
\]
%together with %and since the function $\abs{\,\cdot\,}_{\mathds{R}^m}^2$ is convex, from (\ref{convXM}) and (\ref{alphaif}) it follows that
Thus, we have
\begin{align*}
\Bigl|\tM(t)&-A^{-1}\tX(t)-\int_0^t \tX(s)\,ds+A^{-1}x+\int_0^tA^{-1}\tf(s)\,ds\Bigr|_{\bE}^2\\
&\le 4\left(\Bigl|\tM(t)-\sum_{i=1}^{\bar{N}}\alpha_i\tM_{\bar m+i}(t)\Bigr|_{\bE}^2
+\Bigl|\sum_{i=1}^{\bar{N}}\alpha_iA^{-1}\tX_{m+i}(t)-A^{-1}\tX(t)\Bigr|_{\bE}^2
\right.\\
&\phantom{X(t)}+\,\Bigl|\sum_{i=1}^{\bar{N}}\alpha_i\int_0^t \tX_{m+i}(s)\,ds-\int_0^t\tX(s)\,ds\Bigr|_\bE^2\\
&\phantom{X(t)}+\,\left.\Bigl|\int_0^t A^{-1}\tf(s)\,ds-\sum_{i=1}^{\bar{N}}\alpha_i\int_0^t A^{-1}\tf_{\bar m+i}(s)\,ds\Bigr|_\bE^2\right).
\end{align*}
Then, by (\ref{convXM1}) and (\ref{alphaif1}) it follows that %and by convexity of the function $\abs{\,\cdot\,}_\bE^2$
\begin{align*}
\tExp\Bigl|\tM(t)&-A^{-1}\tX(t)-\int_0^t \tX(s)\,ds+A^{-1}x+\int_0^tA^{-1}\tf(s)\,ds\Bigr|_{\bE}^2\\
&\le 4\left(\sum_{i=1}^{\bar{N}}\alpha_i\tExp\Bigl|\tM(t)-\tM_{m+i}(t)\Bigr|_{\bE}^2
+\sum_{i=1}^{\bar{N}}\alpha_i\norm{A^{-1}}_{\Lin(\bE)}^2\tExp\Bigl|\tX_{\bar m+i}(t)-\tX(t)\Bigr|_\bE^2
\right.\\
&\phantom{X(t)}+\sum_{i=1}^{\bar{N}}\alpha_i\tExp\Bigl|\int_0^t (\tX_{\bar m+i}(s)-\tX(s))\,ds\Bigr|_\bE^2\\
&\phantom{X(t)}+T\norm{A^{-1}}_{\Lin(\bE)}^2\left.\tExp\int_0^T\Bigl|\tf(s)-\sum_{i=1}^{\bar{N}}\alpha_i \tf_{\bar m+i}(s)\Bigr|_{\bE}^2\,ds\right)\\
&\le 4\left(1+\norm{A^{-1}}_{\Lin(\bE)}^2+T+T\norm{A^{-1}}_{\Lin(\bE)}^2\right)\varepsilon.
\end{align*}
Since $\varepsilon>0$ is arbitrary, (\ref{tM2LG}) follows.
\end{proof}
\noindent\textsc{Step 7.} In view of Lemma \ref{tMmartLG} and (\ref{tM2LG}), by the Martingale Representation Theorem \ref{martrep} there exist an extension of the probability space $(\tOmega,\tFil,\tProb),$ which we also denote $(\tOmega,\tFil,\tProb),$ and a $\bH-$cylindrical Wiener process $\{\tilde{W}(t)\}_{t\ge 0}$ defined on $(\tOmega,\tFil,\tProb),$ such that
\[
\tM(t)=\int_0^t A^{-1}\tilde{g}(s)\,d\tilde {W}(s), \ \ \ \tProb-\text{a.s.}, \ \ \ t\in [0,T],
\]
that is, for each $t\in [0,T],$ we have
\[
A^{-1}\tX(t)+\int_0^t\tX(s)\,ds=A^{-1}x_0+\int_0^t A^{-1}\bar F(r,\tX(r),\tq_r)\,dr+\int_0^t A^{-1}G(r,\tX(r))\,d\tilde{W}(r).
\]
By the same argument used in step 3 (cf. Lemma \ref{Lem:strong}) we get
\[
A^{-1}\tX(t)=S_t A^{-1}x_0+\int_0^t S_{t-r}A^{-1}\bar F(r,\tX(r),\tq_r)\,dr+\int_0^t S_{t-r} A^{-1}G(r,\tX(r))\,d\tilde{W}(r)
\]
for each $t\in [0,T].$ Hence,
\[
\tX(t)=S_tx_0+\int^t_0 S_{t-r} \bar F(r,\tX(r),\tq_r)\,dr+\int^t_0 S_{t-r} G(r,\tX(r))\,d\tilde{W}(r), \; \Prob-\text{a.s.}
\]
In other words, $\tilde\pi:=(\tOmega,\tFil,\tProb,\tilde{\mathds{F}},\{\tilde{W}(t)\}_{t\geq 0},\{\tq_t\}_{t\geq 0},\{\tX(t)\}_{t\geq 0})$
is a weak admissible relaxed control. Lastly, by the Fiber Product Lemma \ref{Lem:fibpro} we have
\[
\underline{\delta}_{\tX_n}\otimes\lambda_n\to \underline{\delta}_{\tX}\otimes\lambda, \ \text{ stably in } \ \Y(0,T;\bE\times M), \ \ \tProb-\text{a.s.}
\]
Since $\bE\times M$ is also a metrisable Suslin space, using Lemma \ref{Lem:lscym} and Fatou's Lemma we get
\[
\tExp\int_0^T\!\int_M h(t,\tX(t),u)\,\tlambda(du,dt)
\le\liminf_{n\to\infty}\tExp\int_0^T\!\int_M h(t,\tX_n(t),u)\,\tlambda_n(du,dt)\]
and since $(\tX_n,\tlambda_n)\overset{d}{=}(X_n,\lambda_n)$ it follows that
\begin{align*}
\bar J(\tilde\pi)&=\tExp\int_0^T\!\int_M h(t,\tX(t),u)\,\tlambda(du,dt)+\tExp\varphi(\tX(T))\\
&\le\liminf_{n\to\infty}\Exp^n\int_0^T\!\int_M h(t,X_n(t),u)\,\lambda_n(du,dt)+\liminf_{n\to\infty}\Exp^n\varphi(X_n(T))\\
&\le \liminf_{n\to\infty}\left[\Exp^n\int_0^T\!\int_M h(t,X_n(t),u)\,\lambda_n(du,dt)+\Exp^n\varphi(X_n(T))\right]\\
&=\inf_{\pi\in\mathcal{U}_{\rm ad}^{\rm w}(x_0)}\bar J(\pi),
\end{align*}
that is, $\tilde\pi$ is a weak optimal relaxed control for \textbf{(RCP)}, and this concludes the proof of Theorem \ref{Th:existence2}.
\end{proof}

\appendix
%\numberwithin{equation}{section}\setcounter{equation}{0}

\section{}
\begin{lem}\label{Lem:mildclassical}
Let $-A$ be the generator of a $C_0$-semigroup $(S_t)_{t\geq 0}$ on a Banach space $B$ such that $0\in\rho(A)$ and let $f\in L^1(0,T;B).$ Then the function $z\in \mathcal{C}([0,T];B)$ defined by
\[
z(t):=\int_0^t S_{t-r}f(r)\,dr, \ \ \ \ t\in [0,T],
\]
satisfies
\[
A^{-1}z(t)+\int_0^t z(s)\,ds=\int_0^t A^{-1}f(s),ds, \ \ \ \ t\in [0,T].
\]
\end{lem}
\begin{proof}
From the identity
\[
-A\int_0^t S_rx\,dr=S_tx-x
\]
we have
\[
\int_0^t S_rx\,dr+A^{-1}S_t x=A^{-1}x
\]
and it follows that
\begin{align*}
    \int_0^t z(r)\,dr&=\int_0^t\int_0^rS_{r-s}f(s)\,ds\,dr\\
    &=\int_0^t\int_s^tS_{r-s}f(s)\,dr\,ds\\
    &=\int_0^t\int_0^{t-s}S_{u}f(s)\,du\,ds\\
    &=\int_0^t A^{-1}f(s)\,ds-\int_0^tA^{-1}S_{t-s}f(s)\,ds\\
    &=\int_0^t A^{-1}f(s)\,ds-A^{-1}z(t).
\end{align*}
\end{proof}

\bibliography{biblio}
\end{document}